\documentclass[12pt]{article}

\usepackage[margin=1in]{geometry}
\usepackage{amsmath, amssymb}
\usepackage{amsthm}
\usepackage{booktabs}
\usepackage{algorithm}
\usepackage{algorithmic}
\usepackage{tikz}
\usetikzlibrary{shapes.geometric, arrows.meta, positioning, fit, backgrounds, shadows}
\usepackage{graphicx}
\usepackage[hidelinks]{hyperref}
\usepackage{cleveref}
\usepackage{fancyhdr}
\newtheorem{theorem}{Theorem}[section]
\newtheorem{lemma}[theorem]{Lemma}
\newtheorem{corollary}[theorem]{Corollary}
\newtheorem{assumption}{Assumption}[section]
\crefname{assumption}{Assumption}{Assumptions}
\Crefname{assumption}{Assumption}{Assumptions}
\theoremstyle{remark}
\newtheorem{remark}{Remark}[section]

\newcommand{\fsp}[1]{\mathcal{#1}}
\newcommand{\algo}[1]{\mbox{\textsf{#1}}}
\newcommand{\x}{\boldsymbol{x}}
\newcommand{\dd}{\boldsymbol{d}}
\newcommand{\g}{\boldsymbol{g}}
\newcommand{\G}{\boldsymbol{G}}
\newcommand{\p}{\boldsymbol{p}}
\newcommand{\Hmat}{\boldsymbol{H}}
\newcommand{\I}{\boldsymbol{I}}
\newcommand{\RR}{\mathbb{R}}
\newcommand{\email}[1]{\texttt{#1}}

\title{Parallel Model-Based Derivative-Free Optimization via Rank-Two KKT Updates}
\author{Donghan Wu%
\thanks{University of California, Berkeley (\email{wfwdh@berkeley.edu}).}
\and Pengcheng Xie%
\thanks{Lawrence Berkeley National Laboratory. Corresponding author (\email{pxie@lbl.gov}).}
}
\date{}

\begin{document}

\maketitle

\begin{abstract}
Derivative-free optimization (DFO) addresses unconstrained problems $\min_{\x\in\RR^n} f(\x)$ where $f$ is accessed only through a zeroth-order oracle. Model-based trust-region methods construct underdetermined quadratic interpolation models from $\mathcal{O}(n)$ points and solve a KKT system to determine model parameters, costing $\mathcal{O}(m^3)$ operations and limiting parallel scalability. It is shown that the KKT matrix for the minimum Frobenius norm updating model depends entirely on inner products of shifted coordinates. Reflecting the interpolation set across a single coordinate axis preserves these inner products and changes only one row and column of the KKT matrix, inducing a rank-at-most-two perturbation whose inverse update via the Sherman-Morrison-Woodbury formula costs $\mathcal{O}(n^2)$ when $m=\mathcal{O}(n)$. The reflection is an isometry in centered Euclidean trust regions and preserves the poisedness constant of the interpolation set; together with standard fully linear model-management assumptions this supports first-order global convergence. The mechanism is embedded in a master-worker parallel algorithm with a Truncated Conjugate Gradient subproblem solver. Numerical results on 530 benchmark problems compare performance against an established DFO solver.
\end{abstract}


\vspace{3pt}

\section{Introduction}
\label{sec:intro}

\subsection{Background and Related Work}
In many engineering and scientific settings, objective functions are accessed only through a ``black box'' or zeroth-order oracle. These problems arise in chemical process modeling, complex physical simulations, and hyperparameter tuning, where derivatives are unavailable, noisy, or too expensive to compute. Such unconstrained problems, formulated as $\min_{\x\in\RR^n} f(\x)$, fall under Derivative-Free Optimization (DFO).

Model-based trust-region methods are among the most widely used DFO techniques. To keep evaluation costs manageable, underdetermined quadratic interpolation with $\mathcal{O}(n)$ points has become standard practice. Powell proposed the least Frobenius norm updating principle to ensure uniqueness: among all quadratic models interpolating the data, select the one whose Hessian deviates minimally from the previous model. This principle underpins well-known solvers like \algo{NEWUOA} and \algo{BOBYQA}.

Despite their success, several bottlenecks limit the scalability of traditional model-based DFO methods. First, updating the interpolation set requires solving a KKT system of size $\mathcal{O}(m \times m)$, costing $\mathcal{O}(m^3)$ floating-point operations. Second, while function evaluations can be distributed across workers, model construction typically runs sequentially on a single node, capping parallel speedup. Third, aggressively replacing interpolation points to probe new directions can degrade the poisedness of the set, leading to ill-conditioned KKT systems and inaccurate surrogate models.

In this paper, we develop a structured, parallelizable framework that addresses these bottlenecks. First, we prove that the KKT matrix for the least Frobenius norm updating model is governed by inner products of shifted coordinates. Second, we introduce a coordinate-flipping strategy and show that a single-axis reflection induces a rank-at-most-two perturbation; we give exact expressions for the perturbation vectors, enabling an $\mathcal{O}(n^2)$ update of the inverse matrix when $m=\mathcal{O}(n)$. Third, we show that the flipping operation preserves poisedness in centered Euclidean trust regions, and with standard fully linear model-management assumptions this supports first-order global convergence. Fourth, we design a multi-machine algorithm and evaluate it on standard DFO benchmark problems.

\label{sec:related}

In derivative-free optimization, trust-region frameworks driven by quadratic interpolation models have long held a dominant position. To avoid the $\mathcal{O}(n^2)$ function evaluations needed for fully determined quadratic models, Powell pioneered underdetermined quadratic models in algorithms such as \algo{NEWUOA} and \algo{BOBYQA} \cite{MJDP06}. Using only $\mathcal{O}(n)$ interpolation points, these methods cut evaluation costs substantially. Uniqueness is enforced via the minimum Frobenius norm updating principle, which minimizes the Hessian change between consecutive iterations.

Recent work has explored regularization mechanisms for interpolation models. Xie and Yuan \cite{xieyuannew} proposed a new underdetermined quadratic interpolation model with improved convergence on non-convex landscapes. Regardless of the regularization norm used, however, the model parameters are determined by a KKT system dominated by a dense geometric sub-matrix. This means maintaining polynomial interpolation models carries an $\mathcal{O}(m^3)$ algebraic overhead each time the sample set is updated.

As modern black-box problems grow in complexity, the cost of a single function evaluation has increased sharply. As noted in recent simulation-based design studies (e.g., Campos et al. \cite{campos2026parallel}), parallelization has become a practical requirement rather than an optional enhancement.

Traditional parallel DFO frameworks operate on a Master-Worker topology, distributing batch evaluations of the objective function. This architecture, however, masks a structural mismatch: while evaluations are decentralized, the maintenance of the quadratic model and the factorization of the KKT matrix remain centralized and sequential. As $n$ grows, the $\mathcal{O}(m^3)$ algebraic bottleneck on the master node quickly outweighs the time saved by parallel evaluations.

This tension becomes more acute as architectures evolve toward fully decentralized or multi-agent setups. In recent work on decentralized black-box optimization, Bergou et al. \cite{bergou2025direct} observed that the computational cost of maintaining and synchronizing quadratic models forces state-of-the-art decentralized frameworks to fall back on model-free direct-search methods. This trade-off---sacrificing the convergence speed of model-based methods for scalability---exposes a central limitation in contemporary parallel DFO.

To reduce this cost, lowering the complexity of model construction has drawn considerable attention. In serial algorithms, low-rank updates (e.g., the Sherman-Morrison-Woodbury formula) are applied to matrix factorizations when a single interpolation point is replaced. These mechanisms break down in highly concurrent settings, where multiple exploratory points change simultaneously and the perturbation degrades into a full matrix reconstruction.

Dimensionality reduction via subspaces offers a complementary path. Hare et al. \cite{hare2024expected} proved that exploring random subspaces in derivative-free algorithms guarantees an expected decrease in the objective function. Their result suggests that exploring low-dimensional random subspaces can circumvent full-space computational costs while retaining theoretical guarantees.

Synthesizing these observations, a concrete challenge for parallel DFO takes shape: \textit{Can one design a subspace exploration transformation that satisfies the expected decrease criteria of Hare et al. \cite{hare2024expected}, induces a low-rank algebraic update to bypass the $\mathcal{O}(n^3)$ bottleneck, and preserves the poisedness of the interpolation set?} The framework introduced in the following sections addresses this question.
\subsection{Preliminaries}
\label{sec:preliminaries}

In this section, we formally introduce the unconstrained black-box optimization problem, review the classical derivative-free trust-region framework, and establish the mathematical foundations for underdetermined quadratic interpolation models.

We consider the unconstrained nonlinear optimization problem of the form:
\begin{equation}
\min_{\x \in \RR^n} f(\x),
\end{equation}
where $f: \RR^n \to \RR$ is a computationally expensive black-box objective function. We assume that $f$ is deterministic and possesses a sufficient degree of smoothness (e.g., $f \in C^2$), but its analytical expression is inaccessible, and neither first-order gradients nor second-order Hessian matrices can be obtained directly or via automatic differentiation. Such problems frequently arise in engineering design, hyperparameter tuning, and simulation-based optimization, where each evaluation of $f(\x)$ may require hours of complex PDE solving or physical experiments.

To maintain numerical stability in floating-point arithmetic, derivative-free algorithms rarely evaluate polynomial bases using the absolute coordinates of $\x$. Instead, an arbitrary \textit{base point} $\x_0 \in \RR^n$ is introduced, and all calculations are performed relative to $\x_0$. Typically, $\x_0$ is chosen as the initial starting point or the current best iterate. This coordinate shifting strategy effectively prevents catastrophic cancellation errors when constructing the interpolation matrices, ensuring that the components of $(\x - \x_0)$ remain appropriately scaled.

Traditional line-search methods rely heavily on directional derivatives, making them less suitable for purely zeroth-order oracles. Consequently, modern model-based DFO solvers predominantly rely on the trust-region framework. At the $k$-th iteration, the algorithm maintains a current optimal iterate $\x_k$ and a trust-region radius $\Delta_k > 0$. A local surrogate model $Q_k(\x)$ is constructed to approximate the true objective function $f(\x)$ within the local neighborhood $\mathcal{B}(\x_k, \Delta_k) = \{\x \in \RR^n : \|\x - \x_k\| \le \Delta_k\}$.

To explore the search space, a trial step $\dd_k$ is computed by approximately solving the constrained trust-region subproblem:
\begin{equation}
\label{eq:tr_subproblem}
\min_{\dd \in \RR^n} Q_k(\x_k + \dd) \quad \text{s.t.} \quad \|\dd\| \le \Delta_k.
\end{equation}
The quality of the trial step $\dd_k$ and the accuracy of the surrogate model $Q_k$ are evaluated by the ratio of the actual reduction in the objective function to the predicted reduction in the model:
\begin{equation}
\rho_k = \frac{ared_k}{pred_k} = \frac{f(\x_k) - f(\x_k + \dd_k)}{Q_k(\x_k) - Q_k(\x_k + \dd_k)}.
\end{equation}
The ratio $\rho_k$ governs both the acceptance of the trial point and the update of the trust-region radius. Given parameters $0 < \eta_1 \le \eta_2 < 1$ and $0 < \gamma_{\text{dec}} < 1 < \gamma_{\text{inc}}$: a step with $\rho_k \ge \eta_2$ is highly successful---the trial point is accepted ($\x_{k+1} = \x_k + \dd_k$) and the trust-region radius is expanded or maintained ($\Delta_{k+1} = \gamma_{\text{inc}} \Delta_k$); a step with $\eta_1 \le \rho_k < \eta_2$ is successful---the trial point is accepted but the radius is kept unchanged ($\Delta_{k+1} = \Delta_k$); and a step with $\rho_k < \eta_1$ is unsuccessful---the trial point is rejected ($\x_{k+1} = \x_k$) and the trust region is contracted ($\Delta_{k+1} = \gamma_{\text{dec}} \Delta_k$).

To construct the surrogate model $Q_k(\x)$, we assume it takes the form of a multivariable quadratic polynomial:
\begin{equation}
\label{eq:quad_form}
Q_k(\x) = c + \g^\top (\x-\x_0) + \frac{1}{2}(\x-\x_0)^\top \G (\x-\x_0),
\end{equation}
where $c \in \RR$, $\g \in \RR^n$, and $\G \in \RR^{n \times n}$ is a symmetric matrix. The parameters $c, \g,$ and $\G$ encompass exactly $\frac{1}{2}(n+1)(n+2)$ degrees of freedom.

Let $\fsp{X}_k = \{\x_1, \dots, \x_m\}$ be the current set of distinct interpolation points. The model is forced to match the evaluated function values exactly at these points, leading to the interpolation conditions:
\begin{equation}
\label{eq:interpolation_cond}
Q_k(\x_i) = f(\x_i), \quad \text{for } i = 1, \dots, m.
\end{equation}
If $m = \frac{1}{2}(n+1)(n+2)$, the system is fully determined, and the quadratic model $Q_k(\x)$ is uniquely defined (provided the points in $\fsp{X}_k$ are poised). However, requiring $\mathcal{O}(n^2)$ function evaluations merely to construct a single model is prohibitively expensive for large $n$. 

To circumvent this evaluation bottleneck, modern DFO algorithms operate in the \textit{underdetermined} regime, utilizing significantly fewer points. In practice, the number of interpolation points is chosen such that:
\begin{equation}
n+2 \le m \le \frac{1}{2}(n+1)(n+2).
\end{equation}
A typical choice is $m = 2n+1$. Because $m$ is strictly less than the degrees of freedom of the quadratic model, the linear system defined by \cref{eq:interpolation_cond} possesses infinitely many solutions. Thus, to uniquely identify the model parameters $(\g, \G)$, additional mathematical principles or regularizations must be imposed. This underdetermined nature directly motivates the least Frobenius norm updating strategy.

\label{sec:least_norm}

\label{subsec:model_formulation}
As discussed in \cref{sec:preliminaries}, when the number of interpolation points $m$ satisfies $n+2 \le m \le \frac{1}{2}(n+1)(n+2)$, the interpolation system is underdetermined. To uniquely identify the quadratic model $Q_k$, we adopt Powell's minimum Frobenius norm updating principle. The underlying philosophy is that when updating the interpolation set locally, the true second-order curvature of the objective function rarely exhibits abrupt changes. Thus, the new model should absorb the new function evaluations while minimally altering the Hessian matrix of the previous model $Q_{k-1}$.

Mathematically, this is formulated as the following equality-constrained optimization problem:
\begin{equation}
\label{eq:leastfrob}
\begin{aligned}
\underset{Q}{\operatorname{\min}}\ &\frac{1}{2} \left\Vert\nabla^{2} Q-\nabla^{2} Q_{k-1}\right\Vert_{F}^{2} \\
\text{s.t.} \ \ & Q(\x_i)=f(\x_i), \quad i = 1, \dots, m.
\end{aligned}
\end{equation}
To simplify the algebraic resolution, we define the difference function $D(\x) = Q_{\text{new}}(\x) - Q_{\text{old}}(\x)$. Suppose a single old point $\x_t$ is replaced by a new evaluated point $\x_{\text{new}}$, the problem translates into finding $D(\x)$ such that:
\begin{equation}
\label{eq:miniFrob}
\begin{aligned}
\underset{D}{\operatorname{\min}} \  &\frac{1}{2} \left\Vert \nabla^2 D\right\Vert_F^2 \\ 
\text{s.t.} \ 
& D(\x_{i})=0, \quad \forall i \neq t, \\
& D(\x_{\text{new}})=f(\x_{\text{new}})-Q_{\text{old}}(\x_{\text{new}}).
\end{aligned}
\end{equation}
Let $r_i$ denote the residual on the right-hand side of the interpolation conditions. For any point in the updated set, $D(\x_i) = r_i$. We express the quadratic difference function as $D(\x) = c + \g^\top (\x-\x_0) + \frac{1}{2}(\x-\x_0)^\top \Hmat_D (\x-\x_0)$, where $\Hmat_D = \nabla^2 D$. To solve this, we construct the Lagrangian function:
\begin{equation}
\label{eq:lagrangian}
\mathcal{L}(\Hmat_D, \g, c, \boldsymbol{\lambda}) = \frac{1}{2} \|\Hmat_D\|_F^2 - 2\sum_{i=1}^m \lambda_i \left[ c + \g^\top (\x_i-\x_0) + \frac{1}{2}(\x_i-\x_0)^\top \Hmat_D (\x_i-\x_0) - r_i \right],
\end{equation}
where $\boldsymbol{\lambda} = (\lambda_1, \dots, \lambda_m)^\top \in \RR^m$ are the Lagrange multipliers associated with the interpolation conditions.

\label{subsec:kkt_structure}
The Lagrangian yields structural properties of the difference model. 

\begin{lemma}[Null-Space Constraints and Hessian Structure]
\label{lem:null_space}
For the minimum Frobenius norm updating model, the optimal Lagrange multipliers $\boldsymbol{\lambda}$ must reside within the left null space of the augmented coordinate matrix. Moreover, the optimal Hessian variation $\Hmat_D$ necessarily takes the form of a linear superposition of rank-one outer products.
\end{lemma}
\begin{proof}
By applying the first-order optimality conditions to the Lagrangian \cref{eq:lagrangian}, we differentiate $\mathcal{L}$ with respect to the variables $c$, $\g$, and the symmetric matrix $\Hmat_D$, and set the derivatives to zero:
\begin{align}
    \frac{\partial \mathcal{L}}{\partial c} &= -2\sum_{i=1}^m \lambda_i = 0, \label{eq:kkt_c} \\
    \nabla_{\g} \mathcal{L} &= -2\sum_{i=1}^m \lambda_i (\x_i - \x_0) = \boldsymbol{0}, \label{eq:kkt_g} \\
    \nabla_{\Hmat_D} \mathcal{L} &= \Hmat_D - \sum_{i=1}^m \lambda_i (\x_i - \x_0)(\x_i - \x_0)^\top = \boldsymbol{0}. \label{eq:kkt_H}
\end{align}
\Cref{eq:kkt_c,eq:kkt_g} strictly enforce that $\boldsymbol{\lambda}$ must satisfy $\sum \lambda_i = 0$ and $\sum \lambda_i (\x_i-\x_0) = \boldsymbol{0}$. \Cref{eq:kkt_H} yields the explicit analytical structure of the Hessian update: $\Hmat_D = \sum_{i=1}^m \lambda_i (\x_i - \x_0)(\x_i - \x_0)^\top$, concluding the proof.
\end{proof}

\begin{theorem}[Analytical Construction of the KKT System]
\label{thm:kkt_matrix}
The unknown parameters $(\boldsymbol{\lambda}^\top, c, \g^\top)^\top$ of the difference function $D(\x)$ are uniquely determined by a symmetric block linear system $\boldsymbol{W} \in \RR^{(m+n+1) \times (m+n+1)}$, where the elements of its geometric sub-matrix $\boldsymbol{A} \in \RR^{m \times m}$ strictly rely on squared Euclidean inner products.
\end{theorem}
\begin{proof}
Substituting the analytical expression of $\Hmat_D$ from \cref{lem:null_space} back into the original interpolation condition $D(\x_i) = r_i$, we obtain:
\begin{equation}
    c + \g^\top (\x_i-\x_0) + \frac{1}{2}(\x_i-\x_0)^\top \left[ \sum_{j=1}^m \lambda_j (\x_j - \x_0)(\x_j - \x_0)^\top \right] (\x_i-\x_0) = r_i.
\end{equation}
Rearranging the quadratic term yields a squared inner product:
\begin{equation}
    c + \g^\top (\x_i-\x_0) + \sum_{j=1}^m \lambda_j \left[ \frac{1}{2} \left( (\x_i - \x_0)^\top (\x_j - \x_0) \right)^2 \right] = r_i, \quad \forall i = 1, \dots, m.
\end{equation}
Combining these $m$ linear equations with the $n+1$ null-space boundary conditions derived in \cref{eq:kkt_c,eq:kkt_g}, the global KKT system is explicitly formulated as:
\begin{equation}
\label{eq:BIG_KKT}
\begin{pmatrix}
\boldsymbol{A} & \boldsymbol{X}^{\top} \\
\boldsymbol{X} & \boldsymbol{0}
\end{pmatrix}
\begin{pmatrix}
\boldsymbol{\lambda} \\ c \\ \g
\end{pmatrix} =
\begin{pmatrix}
\boldsymbol{r} \\ 0 \\ \boldsymbol{0}
\end{pmatrix},
\end{equation}
where $\boldsymbol{0}$ is an $(n+1) \times (n+1)$ zero matrix. The augmented coordinate matrix $\boldsymbol{X} \in \RR^{(n+1) \times m}$ has $i$th column $(1,(\x_i-\x_0)^\top)^\top$. The elements of the geometric block $\boldsymbol{A}$ are exactly $\boldsymbol{A}_{i j} =\frac{1}{2}\langle \x_{i}-\x_0, \x_{j}-\x_0 \rangle^{2}$, completing the proof.
\end{proof}

\label{subsec:model_properties}
The derivation in \cref{thm:kkt_matrix} shows that the largest computational component of the KKT matrix (the $m \times m$ dense block $\boldsymbol{A}$) is governed exclusively by inner products. This gives a geometric invariance property:

\begin{theorem}[Orthogonal Invariance of Interpolation Geometry]
\label{thm:orthogonal_invariance}
Let $\boldsymbol{P} \in \RR^{n \times n}$ be any orthogonal transformation matrix (i.e., $\boldsymbol{P}^\top \boldsymbol{P} = \I$). If this orthogonal transformation is uniformly applied to the shifted interpolation set, yielding a new set $\hat{\fsp{X}}$ where $(\hat{\x}_i - \x_0) = \boldsymbol{P}(\x_i - \x_0)$ for all $i$, then the resulting new geometric sub-matrix $\widehat{\boldsymbol{A}}$ is strictly identical to the original matrix $\boldsymbol{A}$.
\end{theorem}
\begin{proof}
According to \cref{thm:kkt_matrix}, the elements of the new geometric matrix are defined as $\widehat{\boldsymbol{A}}_{i j} = \frac{1}{2} ((\hat{\x}_i - \x_0)^\top (\hat{\x}_j - \x_0))^2$. Substituting the orthogonal transformation into the inner product yields:
\begin{equation}
\widehat{\boldsymbol{A}}_{i j} = \frac{1}{2} \left( (\x_i - \x_0)^\top \boldsymbol{P}^\top \boldsymbol{P} (\x_j - \x_0) \right)^2.
\end{equation}
Since $\boldsymbol{P}^\top \boldsymbol{P} = \I$, the inner product evaluates to $\frac{1}{2} ( (\x_i - \x_0)^\top \I (\x_j - \x_0) )^2 = \boldsymbol{A}_{i j}$. Thus, $\widehat{\boldsymbol{A}} = \boldsymbol{A}$ holds unconditionally for all elements.
\end{proof}

\begin{remark}
\label{rem:cliffhanger}
\Cref{thm:orthogonal_invariance} points to a limitation in traditional DFO algorithms as well as an opportunity. When an arbitrary interpolation point is replaced, the pairwise inner products change globally, conventionally requiring an $\mathcal{O}(n^3)$ refactorization of the KKT system. The theorem, however, guarantees that restricting point-generation to \textit{orthogonal transformations} keeps the $\boldsymbol{A}$ block unchanged, confining the perturbation to the boundary block $\boldsymbol{X}$. A dense rotation can change many rows of $\boldsymbol{X}$; the rank-two update below is obtained specifically from a single-coordinate reflection, for which only one coordinate row changes.
\end{remark}

\section{Algorithm}
\label{sec:algorithm}

\subsection{Low-Rank Model Construction}
\label{sec:low_rank_updates}

As established in the previous section, restricting interpolation set modifications to orthogonal transformations ensures the geometric sub-matrix $\boldsymbol{A}$ remains invariant. This confines any structural perturbation to the augmented coordinate boundaries, inducing low-rank updates and avoiding the traditional $\mathcal{O}(n^3)$ refactorization.

When the interpolation set undergoes a structured modification (such as replacing a single point or performing a coordinate flip), the perturbation to the global KKT matrix $\Delta \boldsymbol{W}$ has the symmetric rank-at-most-two structure $\Delta\boldsymbol{W} = \boldsymbol{e}_k \boldsymbol{w}^\top + \boldsymbol{w} \boldsymbol{e}_k^\top$, where $\boldsymbol{e}_k$ is a standard basis vector and $\boldsymbol{w}$ is the perturbation vector.

\begin{theorem}[Analytical $\mathcal{O}(n^2)$ SMW Update]
\label{thm:smw}
Given a rank-at-most-two perturbation on the $k$-th row and column, the updated inverse KKT matrix $\widehat{\boldsymbol{H}} = (\boldsymbol{W} + \Delta\boldsymbol{W})^{-1}$ can be analytically computed via the Sherman-Morrison-Woodbury (SMW) identity:
\begin{equation}
\label{eq:updating-formula}
\widehat{\boldsymbol{H}}=\boldsymbol{H}+\frac{1}{\delta_{\rm SMW}}\left[q\boldsymbol{a}\boldsymbol{a}^{\top}+\alpha\boldsymbol{b}\boldsymbol{b}^{\top}-(1+\tau)\left(\boldsymbol{a}\boldsymbol{b}^{\top}+\boldsymbol{b}\boldsymbol{a}^{\top}\right)\right],
\end{equation}
where $\boldsymbol{a}=\boldsymbol{H}\boldsymbol{e}_k$, $\boldsymbol{b}=\boldsymbol{H}\boldsymbol{w}$, $\alpha=\boldsymbol{e}_k^{\top}\boldsymbol{a}$, $q=\boldsymbol{w}^{\top}\boldsymbol{b}$, $\tau=\boldsymbol{e}_k^{\top}\boldsymbol{b}$, and $\delta_{\rm SMW}=(1+\tau)^2-\alpha q$.
\end{theorem}

\begin{proof}
Let the original KKT matrix be $\boldsymbol{W}$ with its inverse $\boldsymbol{H} = \boldsymbol{W}^{-1}$. The perturbation has the symmetric rank-at-most-two structure $\Delta\boldsymbol{W} = \boldsymbol{e}_k\boldsymbol{w}^\top + \boldsymbol{w}\boldsymbol{e}_k^\top$.
This can be formulated as a block low-rank update $\Delta\boldsymbol{W} = \boldsymbol{U} \boldsymbol{V}^\top$, where $\boldsymbol{U} = [\boldsymbol{e}_k, \boldsymbol{w}]$ and $\boldsymbol{V} = [\boldsymbol{w}, \boldsymbol{e}_k]$.

By applying the general Woodbury matrix identity, the updated inverse $\widehat{\boldsymbol{H}} = (\boldsymbol{W} + \boldsymbol{U} \boldsymbol{V}^\top)^{-1}$ is given by:
\begin{equation}
\label{eq:woodbury_general}
    \widehat{\boldsymbol{H}} = \boldsymbol{H} - \boldsymbol{H} \boldsymbol{U} \boldsymbol{C}^{-1} \boldsymbol{V}^\top \boldsymbol{H},
\end{equation}
where $\boldsymbol{C} \in \RR^{2 \times 2}$ is the capacitance matrix defined as $\boldsymbol{C} = \I_2 + \boldsymbol{V}^\top \boldsymbol{H} \boldsymbol{U}$. Let us explicitly compute the inner products for $\boldsymbol{C}$:
\begin{equation*}
    \boldsymbol{C} = \I_2 + \begin{bmatrix} \boldsymbol{w}^\top \\ \boldsymbol{e}_k^\top \end{bmatrix} \boldsymbol{H} \begin{bmatrix} \boldsymbol{e}_k & \boldsymbol{w} \end{bmatrix} 
    = \begin{bmatrix} 1 + \tau & q \\ \alpha & 1 + \tau \end{bmatrix}.
\end{equation*}

The determinant of the capacitance matrix $\boldsymbol{C}$ dictates the denominator: $\det(\boldsymbol{C})=(1+\tau)^2-\alpha q=\delta_{\rm SMW}$. If this quantity is nonzero, the $2 \times 2$ block system can be inverted algebraically.

Expanding the block matrix product $\boldsymbol{H} \boldsymbol{U} \boldsymbol{C}^{-1} \boldsymbol{V}^\top \boldsymbol{H}$ from \cref{eq:woodbury_general} directly yields the analytical form presented in \cref{eq:updating-formula}. Since evaluating this formula requires one matrix-vector multiplication and rank-two outer products, the inverse matrix $\widehat{\boldsymbol{H}}$ is updated accurately in $\mathcal{O}(n^2)$ operations when $m=\mathcal{O}(n)$.
\end{proof}

A potential issue in the algebraic update \cref{eq:updating-formula} is the division by zero if the scalar $\delta_{\rm SMW} = 0$. We establish that this algebraic singularity is strictly prohibited by the geometric non-degeneracy of the interpolation set.

\begin{theorem}[Non-Singularity of the Capacitance Scalar]
\label{thm:non_singular}
During the rank-at-most-two updating process, assume the original KKT matrix $\boldsymbol{W}$ is nonsingular. Then the SMW denominator scalar satisfies $\delta_{\rm SMW} \neq 0$ if and only if the updated KKT matrix $\widehat{\boldsymbol{W}}$ is nonsingular, equivalently if the updated interpolation set $\widehat{\fsp{X}}$ is poised for the minimum-Frobenius-norm interpolation system. In particular, a coordinate flip from a poised set preserves nonsingularity, so $\delta_{\rm SMW} \neq 0$.
\end{theorem}
\begin{proof}
As derived in \cref{thm:smw}, the perturbation is $\widehat{\boldsymbol{W}} = \boldsymbol{W} + \boldsymbol{U}\boldsymbol{V}^\top$, and the $2 \times 2$ capacitance matrix is $\boldsymbol{C} = \I_2 + \boldsymbol{V}^\top \boldsymbol{H} \boldsymbol{U}$. The Matrix Determinant Lemma yields $\det(\widehat{\boldsymbol{W}}) = \det(\boldsymbol{W}) \det(\boldsymbol{C})$, with $\det(\boldsymbol{C})=\delta_{\rm SMW}$. 

Since the original matrix $\boldsymbol{W}$ is nonsingular ($\det(\boldsymbol{W}) \neq 0$), $\delta_{\rm SMW}=0$ holds if and only if $\det(\widehat{\boldsymbol{W}})=0$. In DFO theory, nonsingularity of this KKT matrix is equivalent to poisedness of the corresponding minimum-Frobenius-norm interpolation set. A coordinate flip is an isometry of the shifted geometry and preserves this nonsingularity, so the denominator cannot vanish for a flip applied to a poised set.
\end{proof}

We now provide the explicit formulations of the perturbation vector $\boldsymbol{w}$ for both the inner-loop single-point replacement and the outer-loop parallel flipping mechanism.

\begin{remark}[Single-Point Replacement Vector]
\label{rem:single_point}
When a single worst point $\x_t$ is replaced by a new trial point $\x_{\text{new}}$ during the inner trust-region search, the perturbation is confined to the $t$-th row and column of $\boldsymbol{W}$. Let $\boldsymbol{y}_i=\x_i-\x_0$ and $\widehat{\boldsymbol{y}}=\x_{\text{new}}-\x_0$. In the representation $\Delta\boldsymbol{W}=\boldsymbol{e}_t\boldsymbol{w}_t^\top+\boldsymbol{w}_t\boldsymbol{e}_t^\top$, the perturbation vector $\boldsymbol{w}_t \in \RR^{m+n+1}$ is the new-minus-old column difference, with the diagonal change halved:
\begin{equation}
\label{eq:w_single}
\begin{cases}
(\boldsymbol{w}_{t})_{i} = \frac{1}{2}\left[(\boldsymbol{y}_i^\top \widehat{\boldsymbol{y}})^2-(\boldsymbol{y}_i^\top \boldsymbol{y}_t)^2\right], \quad & i = 1, \dots, m,\ i\neq t, \\
(\boldsymbol{w}_{t})_{t} = \frac{1}{4}\left(\|\widehat{\boldsymbol{y}}\|^4-\|\boldsymbol{y}_t\|^4\right), \\
(\boldsymbol{w}_{t})_{m+1} = 0, \\
(\boldsymbol{w}_{t})_{m+1+j} = \widehat{\boldsymbol{y}}_j-(\boldsymbol{y}_t)_j, \quad & j = 1, \dots, n.
\end{cases}
\end{equation}
\end{remark}

Building on the inner-product isomorphism, we define a $t$-axis flip as negating the $t$-th coordinate of every interpolation point relative to the base point: $(\x_i-\x_0)^{(t)} \to -(\x_i-\x_0)^{(t)}$.

\begin{theorem}[The Flipping Perturbation Vector]
\label{thm:flip}
Let the interpolation set $\fsp{X}$ undergo a $t$-axis flip. The overall perturbation $\Delta\boldsymbol{W}$ to the KKT matrix is rank at most two on the $k$-th row/column ($k = m+1+t$), and the corresponding perturbation vector $\boldsymbol{w}_k \in \RR^{m+n+1}$ is:
\begin{equation}
\label{eq:w_flip}
\begin{cases}
(\boldsymbol{w}_{k})_{i} = (\hat{\x}_i - \x_0)^{(t)}-(\x_i - \x_0)^{(t)}=-2(\x_i - \x_0)^{(t)}, \quad & i = 1, \dots, m, \\
(\boldsymbol{w}_{k})_{m+1} = 0, \\
(\boldsymbol{w}_{k})_{m+1+j} = 0, \quad & j = 1, \dots, n.
\end{cases}
\end{equation}
\end{theorem}
\begin{proof}
Because the $t$-axis flip is an orthogonal reflection across the hyperplane $x^{(t)}=0$ in shifted coordinates, inner products are conserved. Thus, the top-left geometric block $\boldsymbol{A}$ remains unchanged. The perturbation is isolated to the coordinate row $k=m+1+t$ of the augmented block, where each entry changes from $(\x_i-\x_0)^{(t)}$ to $-(\x_i-\x_0)^{(t)}$. This new-minus-old difference gives the sparse vector $\boldsymbol{w}_k$ defined above.
\end{proof}

A typical vulnerability in $\mathcal{O}(n^2)$ updating algorithms is the ``cold start'' problem: constructing the initial inverse matrix $\boldsymbol{H}_0 = \boldsymbol{W}_0^{-1}$ traditionally demands an $\mathcal{O}(n^3)$ matrix factorization, creating a bottleneck for high-dimensional problems. Our framework avoids this via geometrical structuring.

\begin{theorem}[Analytical Inverse of the Orthogonal Initial System]
\label{thm:analytical_init}
Suppose the initial interpolation set $\fsp{X}_0$ is constructed using a standard orthogonal cross-stencil, i.e., $\fsp{X}_0 = \{\x_0\} \cup \{\x_0 \pm \Delta_{\text{init}} \boldsymbol{e}_j\}_{j=1}^n$. Then, the initial KKT matrix $\boldsymbol{W}_0$ possesses a highly sparse, block-diagonal-like structure, and its exact inverse $\boldsymbol{H}_0 = \boldsymbol{W}_0^{-1}$ can be constructed analytically in $\mathcal{O}(n^2)$ operations.
\end{theorem}
\begin{proof}
Consider the geometric elements $\boldsymbol{A}_{ij} = \frac{1}{2}\langle \x_i-\x_0, \x_j-\x_0 \rangle^2$. For the cross-stencil, vectors $\x_i-\x_0$ and $\x_j-\x_0$ are orthogonal unless they lie on the same coordinate axis. Consequently, $\langle \pm \Delta_{\text{init}} \boldsymbol{e}_u, \pm \Delta_{\text{init}} \boldsymbol{e}_v \rangle = 0$ for $u \neq v$. This forces the $m \times m$ matrix $\boldsymbol{A}$ to become a sparse matrix consisting primarily of zeros and diagonal constant blocks of $\frac{1}{2}\Delta_{\text{init}}^4$. 
Moreover, the coordinate matrix $\boldsymbol{X}$ solely contains elements $0$, $1$, and $\pm \Delta_{\text{init}}$. By recursively applying the Schur complement to this highly patterned, block-sparse structure, the inverse $\boldsymbol{H}_0$ can be explicitly formulated using only diagonal matrix additions and scalar scaling. Thus, $\boldsymbol{H}_0$ is instantiated mathematically without any numerical LU or QR factorization.
\end{proof}

\label{subsec:complexity}
The integration of the results above keeps the proposed framework at $\mathcal{O}(n^2)$ complexity throughout its lifecycle.
For a standard DFO trust-region algorithm, solving the KKT system from scratch at each iteration requires about $\frac{2}{3}(3n)^3 = 18n^3$ floating-point operations (FLOPs).

In our framework, the initial base model is constructed in $\mathcal{O}(n^2)$ FLOPs via \cref{thm:analytical_init}. During the parallel branching phase, each machine generates its flipped model by evaluating the SMW formula (\cref{thm:smw}). The dominant operation in \cref{eq:updating-formula} is the matrix-vector multiplication $\boldsymbol{H}\boldsymbol{w}$, requiring $\mathcal{O}((3n)^2) \approx \mathcal{O}(n^2)$ FLOPs. The inner-loop single-point replacements similarly run in $\mathcal{O}(n^2)$ time. Overall, the model maintenance overhead is $\mathcal{O}(n^2)$.
\subsection{Parallel Framework}
\label{sec:parallel}

Based on the low-rank algebraic structures established in the previous sections, we are now positioned to introduce a new parallel framework for model construction in derivative-free trust-region methods.

\label{subsec:motivation}
In derivative-free optimization, parallelization is traditionally achieved by evaluating the objective function at multiple sample points simultaneously. The core algebraic operations---model construction and trust-region subproblem solving---are, however, sequential processes.

Standard model-based algorithms replace a single interpolation point, rebuild or update the model factorization, and decide the next trial step based on the updated geometry. Attempting to parallelize this naively---distributing a baseline inverse matrix $\boldsymbol{H}$ to $P$ machines and having each explore different points independently---creates a bottleneck. Evaluating uncoordinated new points destroys the prior matrix structure, forcing each machine to solve a new KKT system from scratch. Each worker would incur the $\mathcal{O}(m^3)$ factorization cost. This algebraic overhead would cancel the time saved by parallel evaluations, motivating a structured mechanism that lets multiple machines update models concurrently without $\mathcal{O}(m^3)$ factorizations.

\label{subsec:multi_machine}
To bypass this cost, we propose a decentralized model construction paradigm using the rank-at-most-two updating formula. Each outer iteration begins with a single, globally synchronized state.

The global state consists of the current best base point $\x_{\text{best}}$, the associated baseline interpolation set $\fsp{X}$, the corresponding inverse KKT matrix $\boldsymbol{H}$, and the current trust-region radius $\Delta$. This state is distributed to $P$ parallel machines. Because $\boldsymbol{H}$ requires $\mathcal{O}(n^2)$ storage, broadcasting it imposes modest memory and communication requirements. All $P$ machines start from an identical, well-poised geometric foundation without redundantly factoring the base KKT system.

\label{subsec:flipping_strategy}
To ensure diverse exploration while maintaining $\mathcal{O}(n^2)$ update efficiency, the $P$ machines search in different directions without degrading the shared poised geometry. This is done via a randomized flipping strategy.

Upon receiving the baseline state, each machine $i \in \{1, \dots, P\}$ independently samples a coordinate axis index $t_i \sim \mathcal{U}\{1, \dots, n\}$. The machine applies a $t_i$-axis flip to its local copy of the interpolation geometry, generating $\hat{\fsp{X}}^{(i)}$. If the flipped physical points are used as interpolation points for the original black-box objective, their function values must be evaluated or already available; otherwise the operation is only a coordinate transformation and does not define a valid interpolation model for a general non-symmetric objective. Because this transformation is orthogonal, machine $i$ can apply the rank-at-most-two update (derived in \cref{thm:smw} and \cref{thm:flip}) to compute its local inverse matrix $\boldsymbol{H}^{(i)}$.

This strategy is effective for two reasons. First, the mutually orthogonal axes cause the $P$ machines to explore distinct orthants of the search space, yielding diverse search directions. Second, generating these $P$ local surrogate models costs $\mathcal{O}(P \cdot n^2)$ total operations, avoiding the $\mathcal{O}(P \cdot n^3)$ sequential bottleneck.

\label{subsec:communication}
In distributed HPC environments, frequent communication between nodes incurs network latency that can offset the gains from parallel processing. To improve the computation-to-communication ratio, our mechanism uses an inner-loop strategy.

After the initial flipping phase, each machine operates asynchronously. For a predefined number of inner steps (e.g., $S=10$), each machine independently solves its local trust-region subproblems, evaluates the objective function, and executes single-point rank-at-most-two updates on its local matrix $\boldsymbol{H}^{(i)}$. During these $S$ inner iterations, no network communication occurs between workers.

A global synchronization barrier is invoked only at the boundaries of the outer iterations. The network executes a global reduction operation (e.g., \texttt{MPI\_Allreduce}) to compare the local best function values $f(\x_{\text{opt}}^{(i)})$ found by all workers. The machine with the lowest objective value is declared the winner. The master node extracts the winning machine's local best point, its updated inverse matrix $\boldsymbol{H}^{(i^*)}$, and its current trust-region radius, establishing them as the new global baseline state. This state is then broadcast to initiate the next outer parallel iteration.

In this section, we describe the implementation of the Parallel Flipping Trust-Region framework. The architecture is a two-tier nested structure: an outer parallel loop for global coordination and diverse geometric exploration, and an inner sequential loop for local trust-region searches to hide network latency.

\label{subsec:overall_framework}
The execution life-cycle proceeds as follows. The master node constructs a base interpolation set and computes the initial inverse KKT matrix analytically. The global state is broadcast to $P$ independent worker machines. During the outer loop, each worker derives a candidate interpolation geometry via a randomized axis flip and the $\mathcal{O}(n^2)$ Sherman-Morrison-Woodbury (SMW) update, after obtaining any objective values required at newly introduced physical points. Each worker then enters an inner loop, executing $S$ sequential trust-region steps---solving local subproblems, evaluating the objective function, and performing single-point rank-at-most-two updates. A global synchronization barrier gathers the best local solutions and establishes the new baseline state for the next outer iteration.

\label{subsec:parallel_model}
To avoid the $\mathcal{O}(n^3)$ factorization at startup, the initial interpolation set $\fsp{X}_0$ is constructed using the orthogonal cross-stencil: $\fsp{X}_0 = \{\x_0\} \cup \{\x_0 \pm \Delta_{\text{init}} \boldsymbol{e}_j\}_{j=1}^n$. As established in \cref{thm:analytical_init}, this geometric symmetry allows the initial inverse matrix $\boldsymbol{H}_0$ to be computed exactly via analytical block inversion. 

At the onset of each outer iteration, machine $i \in \{1, \dots, P\}$ samples a random axis $t_i \sim \mathcal{U}\{1, \dots, n\}$ and negates the corresponding coordinates of its interpolation geometry. Utilizing the explicit perturbation vector $\boldsymbol{w}_{m+1+t_i}$ defined in \cref{eq:w_flip}, machine $i$ applies the SMW update to acquire the inverse KKT matrix for its local geometry in $\mathcal{O}(n^2)$ time when $m=\mathcal{O}(n)$. The resulting interpolation model is valid only after the associated objective values for the local interpolation set have been supplied.

\label{subsec:tcg_solver}
To compute the trial step $\x_{\text{new}} = \x_{\text{opt}} + \p_k$ within the local trust region $\|\p\| \le \Delta^{(i)}$, we employ the Steihaug-Toint Truncated Conjugate Gradient (TCG) method. 

Our flipping mechanism works naturally with the TCG solver: TCG never requires explicit matrix inversion, relying only on Hessian-vector products ($\nabla^2 Q \cdot \boldsymbol{v}$). Since our framework maintains the quadratic model efficiently, each TCG iteration runs in $\mathcal{O}(n^2)$ operations when $m=\mathcal{O}(n)$. The full subproblem cost is therefore $\mathcal{O}(N_{\rm CG}n^2)$, where $N_{\rm CG}$ is the number of TCG iterations used.

A critical step in the TCG method is enforcing the trust-region boundary constraint. If a conjugate search direction $\dd_k$ causes the step to violate the trust region, or if negative curvature is encountered ($\dd_k^\top \nabla^2 Q \dd_k \le 0$), the search is truncated at the boundary. Geometrically, this requires finding the positive scalar $\alpha^* > 0$ such that $\|\p_k + \alpha^* \dd_k\|^2 = \Delta^2$. This geometric intersection naturally expands to a quadratic equation:
\begin{equation}
\label{eq:truncation_quadratic}
(\dd_k^\top \dd_k)(\alpha^*)^2 + 2(\p_k^\top \dd_k)\alpha^* + (\p_k^\top \p_k - \Delta^2) = 0.
\end{equation}
The exact analytical positive root for the truncation step size is thus immediately given by:
\begin{equation}
\label{eq:truncation}
\alpha^* = \frac{-\p_k^\top \dd_k + \sqrt{(\p_k^\top \dd_k)^2 - (\dd_k^\top \dd_k)(\p_k^\top \p_k - \Delta^2)}}{\dd_k^\top \dd_k}.
\end{equation}

\label{subsec:acceptance}
Upon evaluating the objective function at $\x_{\text{new}}$, the local machine determines the step's viability using the standard reduction ratio:
\begin{equation}
\label{eq:rho}
\rho = \frac{f(\x_{\text{opt}}) - f(\x_{\text{new}})}{Q_s(\x_{\text{opt}}) - Q_s(\x_{\text{new}})}.
\end{equation}
We define the acceptance thresholds $\eta_1 = 0.25$, $\eta_2 = 0.75$, and an expansion factor $\gamma = 2$. 
If $\rho \ge \eta_2$, the model accuracy is deemed highly reliable, the trial point is accepted, and the trust-region radius is expanded: $\Delta \gets \min(\gamma \Delta, \Delta_{\text{max}})$. 
If $\eta_1 \le \rho < \eta_2$, the step is accepted, and the radius is typically maintained. 
If $\rho < \eta_1$, the trial point is rejected, and the radius is contracted: $\Delta \gets \max(\Delta/\gamma, \Delta_{\text{min}})$.

The enforcement of the safety bounds $\Delta \in [\Delta_{\text{min}}, \Delta_{\text{max}}]$ (e.g., $[10^{-12}, 10^6]$) prevents floating-point underflow or divergence during deep searches in highly non-convex black-box landscapes. Whenever a step is accepted, the worst-performing point in the interpolation set is replaced by $\x_{\text{new}}$, and a single-point rank-at-most-two update (\cref{rem:single_point}) is executed to refresh the local KKT matrix.

\label{subsec:complete_algorithm}
The integration of the parallel flipping initialization, TCG inner loop, and synchronized state updates culminates in the formal procedure presented in \cref{alg:framework}.

\begin{algorithm}[htbp]
\caption{Parallel Trust-Region with Flipping Mechanism}
\label{alg:framework}
\begin{algorithmic}[1]
\REQUIRE Dimension $n$, machines $P$, initial point $\x_{0}$, inner steps $S=10$.
\STATE \textbf{Hyperparameters:} $\gamma = 2, \eta_1 = 0.25, \eta_2 = 0.75$, bounds $\Delta \in [10^{-12}, 10^6]$.
\STATE \textbf{Initialization:} Evaluate $f(\x)$ on cross-stencil $\fsp{X} = \{\x_0\} \cup \{\x_0 \pm \Delta_{\text{init}} \boldsymbol{e}_j\}_{j=1}^n$. Compute $\boldsymbol{H} \gets \boldsymbol{W}_0^{-1}$ analytically (\cref{thm:analytical_init}).
\WHILE{termination criteria not met}
    \FOR{machine $i=1,\dots,P$ \textbf{in parallel}}
        \STATE \textbf{[Flip]} Uniformly select axis $t_i \in \{1,\dots,n\}$. Generate $\hat{\fsp{X}}^{(i)}$ via $t_i$-axis flip and obtain any missing objective values on the flipped interpolation points.
        \STATE \textbf{[Update]} Compute perturbation $\boldsymbol{w}_{m+1+t_i}$ (\cref{eq:w_flip}). Update inverse $\boldsymbol{H}^{(i)}$ via $\mathcal{O}(n^2)$ SMW formula (\cref{eq:updating-formula}).
        \STATE Solve KKT system for $(\boldsymbol{\lambda}^{(i)}, c^{(i)}, \g^{(i)})$; Construct local quadratic model $Q_0^{(i)}(\x)$.
        
        \STATE \textbf{[Inner Loop]} Initialize local state: $\boldsymbol{H}_s \gets \boldsymbol{H}^{(i)}$, $Q_s \gets Q_0^{(i)}$, $\Delta^{(i)} \gets \Delta$.
        \FOR{$s=1,\dots,S$}
            \STATE Identify local center: $\x_{\text{opt}} \gets \arg\min_{\x \in \hat{\fsp{X}}^{(i)}} f(\x)$.
            \STATE Solve subproblem: $\x_{\text{new}} \gets \text{TCG}(\x_{\text{opt}}, \nabla Q_s, \nabla^2 Q_s, \Delta^{(i)})$ applying \cref{eq:truncation}.
            \STATE Evaluate actual vs. predicted reduction ratio $\rho$ using \cref{eq:rho}.
            \STATE \textbf{TR Update:} If $\rho \ge \eta_2$, $\Delta^{(i)} \gets \min(\gamma \Delta^{(i)}, 10^6)$; If $\rho < \eta_1$, $\Delta^{(i)} \gets \max(\Delta^{(i)}/\gamma, 10^{-12})$.
            \IF{$\rho \ge \eta_1$ (Step accepted)}
                \STATE Replace the worst point index $t$ in $\hat{\fsp{X}}^{(i)}$ with $\x_{\text{new}}$.
                \STATE Compute perturbation $\boldsymbol{w}_t$ (\cref{eq:w_single}). Execute $\boldsymbol{H}_s \gets \text{Rank2Up}(\boldsymbol{H}_s, t, \boldsymbol{w}_t)$.
                \STATE Update local model to $Q_s(\x)$ using new KKT residuals.
            \ENDIF
        \ENDFOR
        \STATE Save local optimal point $\x_{\text{opt}}^{(i)}$, and store matrix state $\boldsymbol{H}^{(i)} \gets \boldsymbol{H}_s$.
    \ENDFOR
    \STATE \textbf{[Global Sync]} Identify global best: $\text{best} \gets \arg\min_i f(\x_{\text{opt}}^{(i)})$ via \texttt{MPI\_Allreduce}.
    \STATE Broadcast the winning matrix state $\boldsymbol{H}^{(\text{best})}$ and radius $\Delta^{(\text{best})}$ to all machines.
\ENDWHILE
\end{algorithmic}
\end{algorithm}
\section{Theoretical Analysis}
\label{sec:theory}

In this section, we establish the theoretical foundations of the proposed Parallel Flipping Trust-Region framework. We prove the geometric preservation property of the flipping mechanism and state the standard additional assumptions under which first-order global convergence follows. We also analyze the computational complexity of the low-rank updates.

\label{subsec:model_accuracy}
To formally analyze the accuracy of the surrogate models, we establish the standard derivative-free optimization (DFO) assumptions regarding the objective function and the algorithm mechanism.

\begin{assumption}[Smoothness and Boundedness]
\label{ass:smoothness}
The objective function $f: \RR^n \to \RR$ is bounded below. The gradient $\nabla f(\x)$ is Lipschitz continuous with a constant $L > 0$ in an open domain containing all iterates.
\end{assumption}

\begin{assumption}[Bounded Model Hessian]
\label{ass:hessian}
The sequence of Hessian matrices of the constructed quadratic models is uniformly bounded, i.e., there exists a constant $\kappa_B > 0$ such that $\|\nabla^2 Q_k\| \le \kappa_B$ for all iterations $k$.
\end{assumption}

\begin{assumption}[Baseline Model Accuracy]
\label{ass:accuracy}
Following the standard framework of DFO, we assume that the underlying sequential single-point update mechanism asymptotically maintains the interpolation set within a neighborhood of the trust region, yielding \textit{fully linear} baseline models.
\end{assumption}

\begin{theorem}[Preservation of Poisedness]
\label{thm:poisedness}
If the original interpolation set $\fsp{X}$ is $\Lambda$-poised in a Euclidean ball centered at the base point $\x_0$, then the flipped interpolation set $\hat{\fsp{X}}$ generated by the $t$-axis flip about $\x_0$ maintains the same poisedness constant $\Lambda$ in that ball.
\end{theorem}
\begin{proof}
Let $F_t=\I-2\boldsymbol{e}_t\boldsymbol{e}_t^\top$ and $S=\operatorname{diag}(1,F_t)$. If $\boldsymbol{X}$ is the augmented coordinate block in \cref{eq:BIG_KKT}, then the flipped block is $\widehat{\boldsymbol{X}}=S\boldsymbol{X}$, while the geometric block $\boldsymbol{A}$ is unchanged. Hence
\begin{equation*}
\widehat{\boldsymbol{W}}=\begin{pmatrix}\I_m&0\\0&S\end{pmatrix}\boldsymbol{W}\begin{pmatrix}\I_m&0\\0&S^\top\end{pmatrix}.
\end{equation*}
The congruence matrix is orthogonal, so nonsingularity and the two-norm condition number of the KKT interpolation system are preserved. Equivalently, the Lagrange polynomials transform as $\widehat{\ell}_i(\x_0+F_t\boldsymbol{y})=\ell_i(\x_0+\boldsymbol{y})$. Since the Euclidean ball centered at $\x_0$ is invariant under $F_t$, their suprema over the ball, and hence the poisedness constant, are identical.
\end{proof}

\begin{remark}
In traditional minimum Frobenius norm updating schemes, unbounded Hessian growth (violating \cref{ass:hessian}) is typically caused by poorly poised interpolation sets where the KKT system becomes ill-conditioned. As proven in \cref{thm:poisedness}, the flipping mechanism does not worsen this conditioning. This removes one geometric source of ill-conditioning, but the bounded-Hessian assumption still depends on the usual model-management mechanism and the sampled function values.
\end{remark}

Based on these properties, we establish the fundamental error bound for our parallel framework.

\begin{lemma}[Inheritance of Fully Linear Bounds]
\label{lem:fully_linear}
Suppose \cref{ass:smoothness} and \cref{ass:accuracy} hold, suppose the base point for the flip is the current trust-region center ($\x_0=\x_k$), and suppose that each flipped model interpolates objective values evaluated at the flipped points. Because the flipping operation preserves the geometric poisedness constant $\Lambda$ in the centered trust region, the standard DFO interpolation error bounds give the fully linear model property for $Q_k^{(i)}$ on $\mathcal{B}(\x_k, \Delta_k)$. That is, there exist constants $\kappa_{ef}, \kappa_{eg} > 0$ independent of $k$, such that for all trial steps $\|\dd\| \le \Delta_k$:
\begin{align}
    |f(\x_k+\dd) - Q_k^{(i)}(\x_k+\dd)| &\le \kappa_{ef} \Delta_k^2, \label{eq:err_f} \\
    \|\nabla f(\x_k) - \nabla Q_k^{(i)}(\x_k)\| &\le \kappa_{eg} \Delta_k. \label{eq:err_g}
\end{align}
\end{lemma}

\subsection{Convergence}
\label{subsec:convergence}
We now establish first-order global convergence. The algebraic flipping update preserves geometry, while convergence additionally requires the usual trust-region step-acceptance and incumbent-update conditions.

The TCG solver used to compute the trial step satisfies a standard sufficient-decrease property.

\begin{lemma}[Sufficient Decrease from TCG]
\label{lem:tcg_decrease}
Under a fully linear model on $\mathcal{B}(\x_k,\Delta_k)$, the trial step $\dd_k$ produced by the TCG solver satisfies
\begin{equation}
    pred_k \;:=\; Q_k(\x_k) - Q_k(\x_k+\dd_k) \;\ge\; \frac{\kappa_{\mathrm{cauchy}}}{2}\|\g_k\|\min\left(\Delta_k,\;\frac{\|\g_k\|}{\|\nabla^2 Q_k\|}\right),
\end{equation}
where $\g_k=\nabla Q_k(\x_k)$ and $\kappa_{\mathrm{cauchy}}\in(0,1]$. In particular, using the bounded model Hessian from \cref{ass:hessian}, the TCG solver guarantees at least the Cauchy decrease
\begin{equation}
    pred_k \ge \frac{1}{2}\|\g_k\|\min\left(\Delta_k,\frac{\|\g_k\|}{\kappa_B}\right).
\end{equation}
\end{lemma}

\begin{assumption}[Model management and incumbent monotonicity]
\label{ass:model_management}
Successful steps satisfy $f(\x_k)-f(\x_{k+1})\ge \eta_1 pred_k$, and the synchronized global incumbent is chosen so that $f(\x_{k+1})\le f(\x_k)$ for all $k$.
\end{assumption}

\begin{theorem}[First-Order Global Convergence]
\label{thm:global_convergence}
Let the sequence of global incumbents $\{\x_k\}$ be generated by the Parallel Flipping Trust-Region algorithm under \cref{ass:smoothness,ass:hessian,ass:accuracy,ass:model_management}. Then
\begin{equation}
    \liminf_{k\to\infty}\|\nabla f(\x_k)\|=0.
\end{equation}
\end{theorem}
\begin{proof}
Assume for contradiction that there is an $\epsilon>0$ such that $\|\nabla f(\x_k)\|\ge\epsilon$ for all sufficiently large $k$. For any iteration with
\begin{equation*}
\Delta_k\le \bar\Delta := \min\left\{\frac{\epsilon}{2\kappa_{eg}},\frac{\epsilon}{2\kappa_B},\frac{(1-\eta_2)\epsilon}{8\kappa_{ef}}\right\},
\end{equation*}
full linearity gives $\|\g_k\|\ge\epsilon/2$, and the Cauchy decrease from \cref{lem:tcg_decrease} yields $pred_k\ge \epsilon\Delta_k/4$. Therefore,
\begin{equation*}
|\rho_k-1|\le\frac{|f(\x_k)-Q_k(\x_k)|+|f(\x_k+\dd_k)-Q_k(\x_k+\dd_k)|}{pred_k}
\le \frac{8\kappa_{ef}\Delta_k}{\epsilon}\le 1-\eta_2.
\end{equation*}
Thus every sufficiently small-radius iteration is highly successful and cannot trigger a radius contraction. Combining this small-radius success property with the standard trust-region radius-update argument gives infinitely many successful iterations whose radii are bounded below by a positive constant. On each such iteration, \cref{ass:model_management} gives an actual decrease bounded below by a positive constant proportional to $\epsilon$. This contradicts monotonicity of $f(\x_k)$ and the lower boundedness of $f$. Hence the contradiction hypothesis is false, proving the result.
\end{proof}

A standard consequence is an upper bound on the number of iterations to reach an $\varepsilon$-stationary point.

\begin{corollary}[Iteration Complexity]
\label{cor:complexity}
Under the assumptions of \cref{thm:global_convergence}, for any $\varepsilon>0$ there exists $K=O\!\left(\varepsilon^{-2}\right)$ such that the algorithm reaches an iterate $\x_k$ with $\|\nabla f(\x_k)\|\le\varepsilon$ within $K$ successful iterations.
\end{corollary}
\begin{proof}
Following the argument of \cref{thm:global_convergence}, each successful iteration with $\Delta_k\ge \bar\Delta(\epsilon)$ reduces the objective by at least $\eta_1\epsilon\bar\Delta/4$. Since $f$ is bounded below (\cref{ass:smoothness}), the number of such iterations cannot exceed $K = O(\epsilon^{-2})$. By \cref{ass:model_management} the iteration sequence contains infinitely many successful steps; the result follows by choosing $\epsilon$ as the desired tolerance.
\end{proof}

These results fit within the DFO convergence theory of Conn, Scheinberg, and Vicente~\cite{Conn2009}: the algorithm inherits standard first-order guarantees because the flipping operation preserves the poisedness constant (\cref{thm:poisedness}) and the fully linear model bounds (\cref{lem:fully_linear}), while the TCG solver supplies the sufficient-decrease property required for the trust-region argument.

\subsection{Complexity}
\label{subsec:complexity_analysis}
We analyze the computational, memory, and communication costs in a distributed setting.

Consider the dimension of the KKT system $p = m+n+1$. With $m \approx 2n+1$, solving the dense KKT system from scratch using LU or QR factorization requires approximately $\frac{2}{3}p^3 \approx 18n^3$ floating-point operations (FLOPs). Our framework avoids this cubic cost across all three phases. In the initialization phase, constructing the initial interpolation set with an orthogonal cross-stencil yields a highly sparse block-diagonal KKT matrix; its inverse $\boldsymbol{H}_0$ is computed analytically using vector scaling and diagonal matrix additions, requiring $\mathcal{O}(n^2)$ FLOPs. In the model updating phase, during parallel branching (flipping) and sequential single-point replacement, the matrix inverse is maintained via the Sherman-Morrison-Woodbury identity; the dominant cost is the matrix-vector multiplication $\boldsymbol{H}\boldsymbol{w}$ at $2p^2 \approx 18n^2$ FLOPs, with rank-2 outer products also in $\mathcal{O}(n^2)$. In the subproblem solving phase, the TCG method relies solely on Hessian-vector multiplications ($\nabla^2 Q \cdot \boldsymbol{v}$) and vector dot products, completing each TCG iteration in $\mathcal{O}(n^2)$ operations when $m=\mathcal{O}(n)$. Consequently, if $N_{\rm CG}$ denotes the number of TCG iterations used per subproblem, the aggregate computational complexity per machine per inner iteration is $\mathcal{O}(N_{\rm CG} n^2)$, and per outer iteration it is $\mathcal{O}(S N_{\rm CG} n^2)$. The latter reduces to $\mathcal{O}(S n^2)$ when $N_{\rm CG}$ is bounded independently of $n$ by truncation. For concreteness, at $n=100$ and $n=1000$, factoring the KKT system from scratch costs $18n^3 \approx 1.8\times 10^7$ and $1.8\times 10^{10}$ FLOPs, respectively. Under our framework, a single matrix-vector update costs $18n^2 \approx 1.8\times 10^5$ and $1.8\times 10^7$ FLOPs, a reduction factor of $n$. At $n=1000$, this is a three-order-of-magnitude saving per update.

The KKT inverse matrix $\boldsymbol{H}$ is the dominant data structure, requiring $\mathcal{O}((m+n+1)^2) = \mathcal{O}(n^2)$ storage per machine. A single $\boldsymbol{H}$ of size $p\times p$ occupies roughly $8p^2$ bytes in double precision (about $0.7$\,MB for $n=100$ and $72$\,MB for $n=1000$). Each of the $P$ machines stores a local copy $\boldsymbol{H}^{(i)}$, so total memory scales by $P$; at $P\le 8$ and $n\le 1000$, the aggregate footprint stays under $1$\,GB. This per-machine memory cost is identical to the serial case.

In the distributed setting, the outer iteration consists of an initial broadcast of the baseline state and a final gather of the best local solution. Broadcasting $\boldsymbol{H}$ and the trust-region radius $\Delta$ costs $\mathcal{O}(p^2 + 1) = \mathcal{O}(n^2)$ data per outer iteration. The gather operation (e.g., \texttt{MPI\_Allreduce}) transmits $\mathcal{O}(1)$ data---a single scalar for the function value and a vector of length $n$ for the best point. The inner-loop strategy with $S=10$ amortizes this communication cost: each machine performs $S$ local trust-region steps and $S$ matrix updates between synchronizations, so the communication-to-computation ratio is $\mathcal{O}(n^2 / (S N_{\rm CG} n^2)) = \mathcal{O}(1 / (S N_{\rm CG}))$. For $S=10$ and typical $N_{\rm CG}\approx 10$, communication accounts for roughly $1\%$ of the outer-iteration cost.
\section{Numerical Results}
\label{sec:experiments}

To evaluate the practical performance of the proposed parallel flipping mechanism, we conduct extensive numerical experiments on a standard derivative-free optimization benchmark set.

We compare the proposed algorithm under different parallel configurations ($P \in \{2, 4, 8\}$) against QARSTA, an established derivative-free optimization solver. All experiments are executed on a Windows workstation equipped with an Intel Core Ultra 9 285H CPU and 32GB RAM. The algorithms are implemented in Python 3.12, utilizing NumPy for core matrix operations and low-rank updates.

To ensure a fair comparison, the maximum computational budget for each test problem is strictly limited to $N_{\max} = 100(n+1)$ function evaluations, where $n$ is the problem dimension. An algorithm is considered to have successfully solved an instance if the normalized function reduction satisfies
\begin{equation}
\label{eq:f_acc}
f_{\mathrm{acc}}^{N} := \frac{f(\x_{0}) - f(\x_{N})}{f(\x_{0}) - f(\x_{\text{best}})} \ge 1-\tau,
\end{equation}
where $\x_{0}$ is the starting point, $\x_{N}$ is the best point found after $N$ evaluations, and $\x_{\text{best}}$ is the known optimum. We evaluate the solvers across four tolerance levels: $\tau \in \{10^{-1}, 10^{-2}, 10^{-3}, 10^{-4}\}$.

The test set comprises 530 problems from the well-known Mor\'{e} and Wild derivative-free optimization benchmark collection. This suite covers dimensions ranging from $n=2$ to $100$ and includes a diverse array of objective function landscapes, featuring smooth, non-smooth, and noisy variations (e.g., deterministic noise with variance $10^{-4}$).

We aggregate the experimental results using two standard tools. Performance profiles $\rho(\alpha)$ measure the relative efficiency of an algorithm as the fraction of problems for which its number of function evaluations is within a factor $\alpha$ of the best-performing solver. Data profiles $\delta(\beta)$ illustrate budget scalability as the fraction of problems solved within a normalized computational budget $\beta = N_f/(n+1)$.

\subsection{Setup and Visualization}
\label{sec:visualization}

To provide an intuitive illustration of the spatial sampling and convergence behavior of the Parallel Flipping algorithm, we conduct a brief visualization test on the foundational 2D Sphere function ($f(\x) = x_1^2 + x_2^2$). The experiment is configured with a parallel dimension of $P=4$, an initial point located at $(4.0, 4.0)$, and an initial trust-region radius $\Delta_0 = 2.0$.

\begin{figure}[htbp]
  \centering
  \begin{minipage}{0.45\textwidth}
    \centering
    \includegraphics[width=\linewidth]{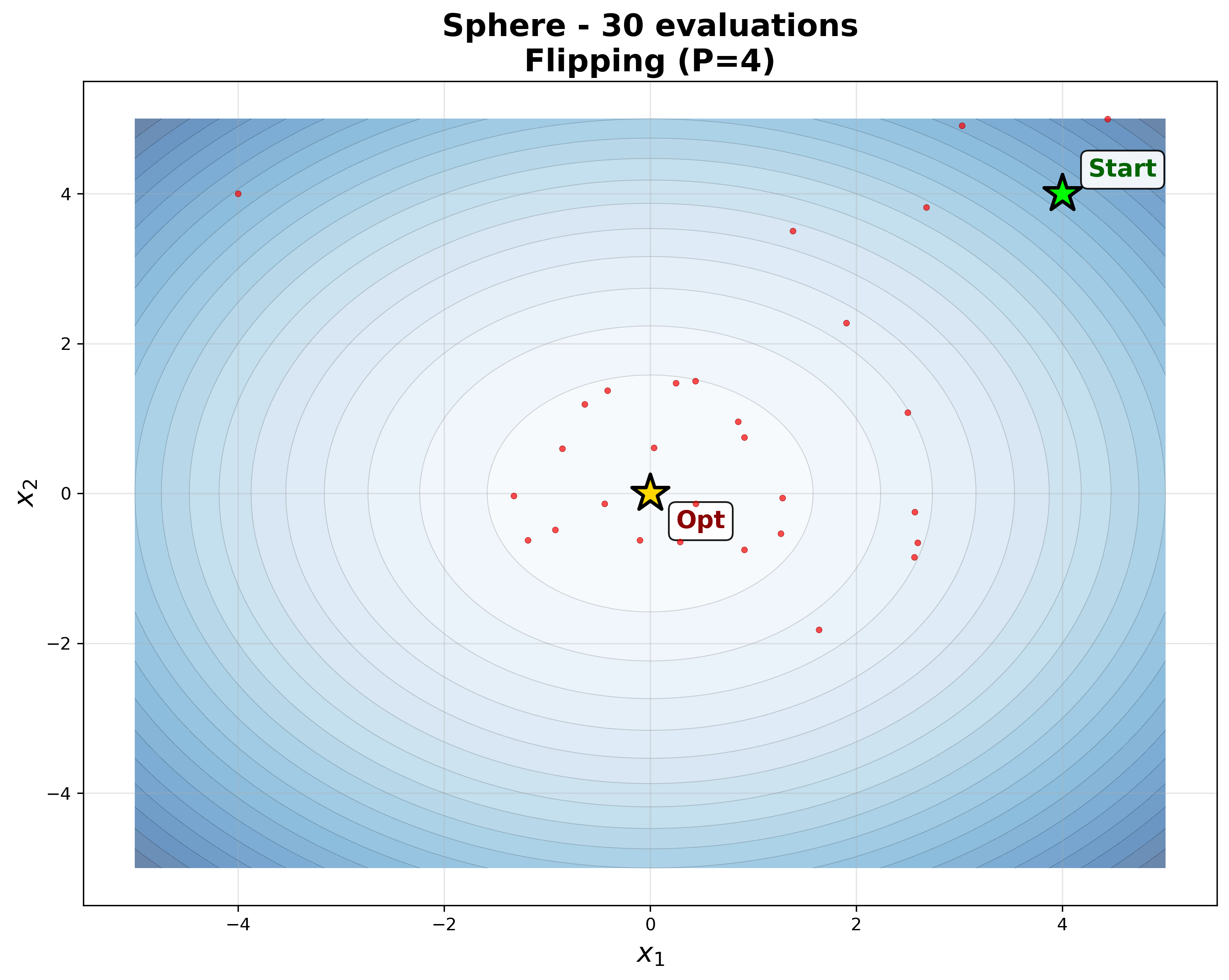}
  \end{minipage}\hfill
  \begin{minipage}{0.45\textwidth}
    \centering
    \includegraphics[width=\linewidth]{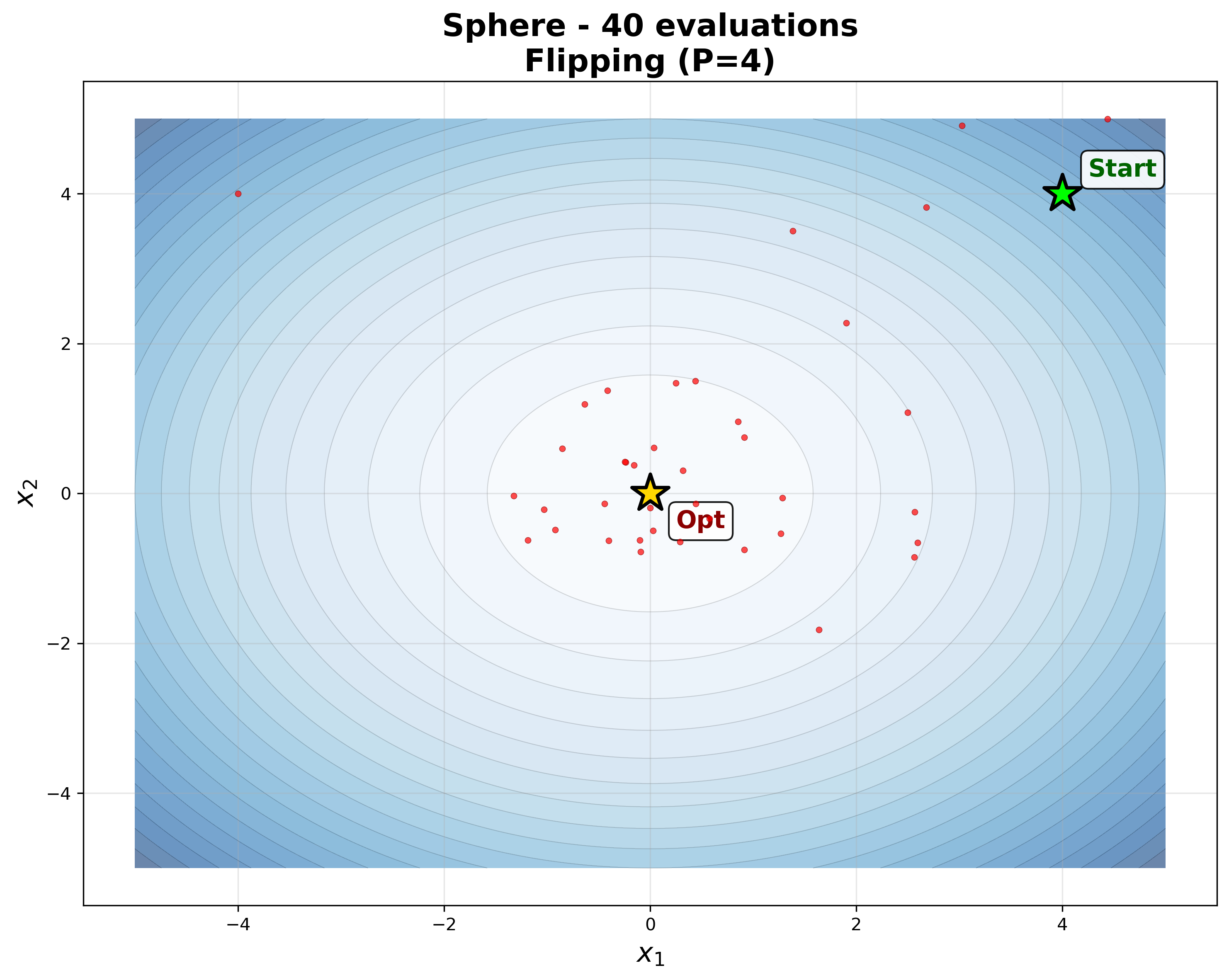}
  \end{minipage}
  \vspace{0.1cm}
  \begin{minipage}{0.45\textwidth}
    \centering
    \includegraphics[width=\linewidth]{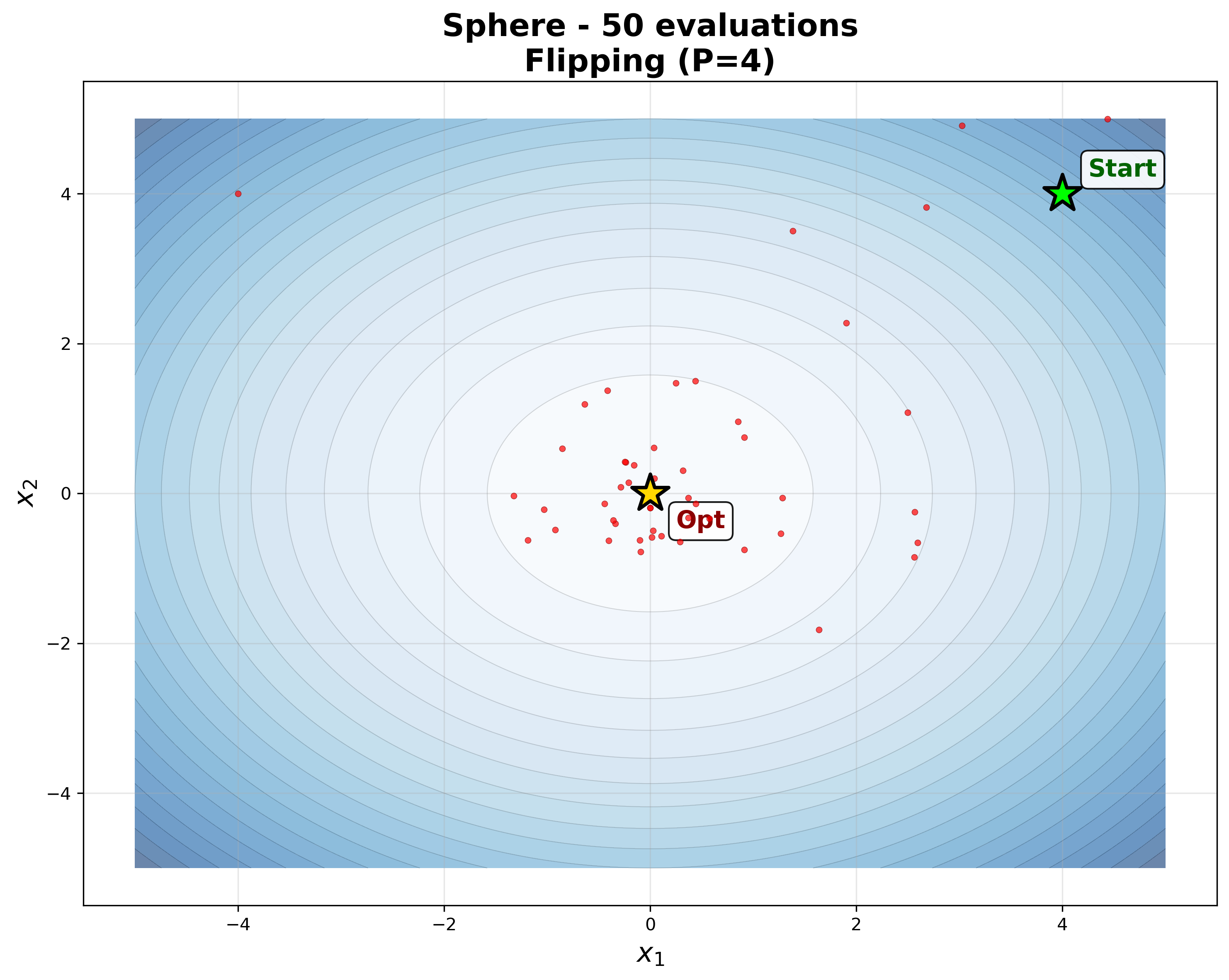}
  \end{minipage}\hfill
  \begin{minipage}{0.45\textwidth}
    \centering
    \includegraphics[width=\linewidth]{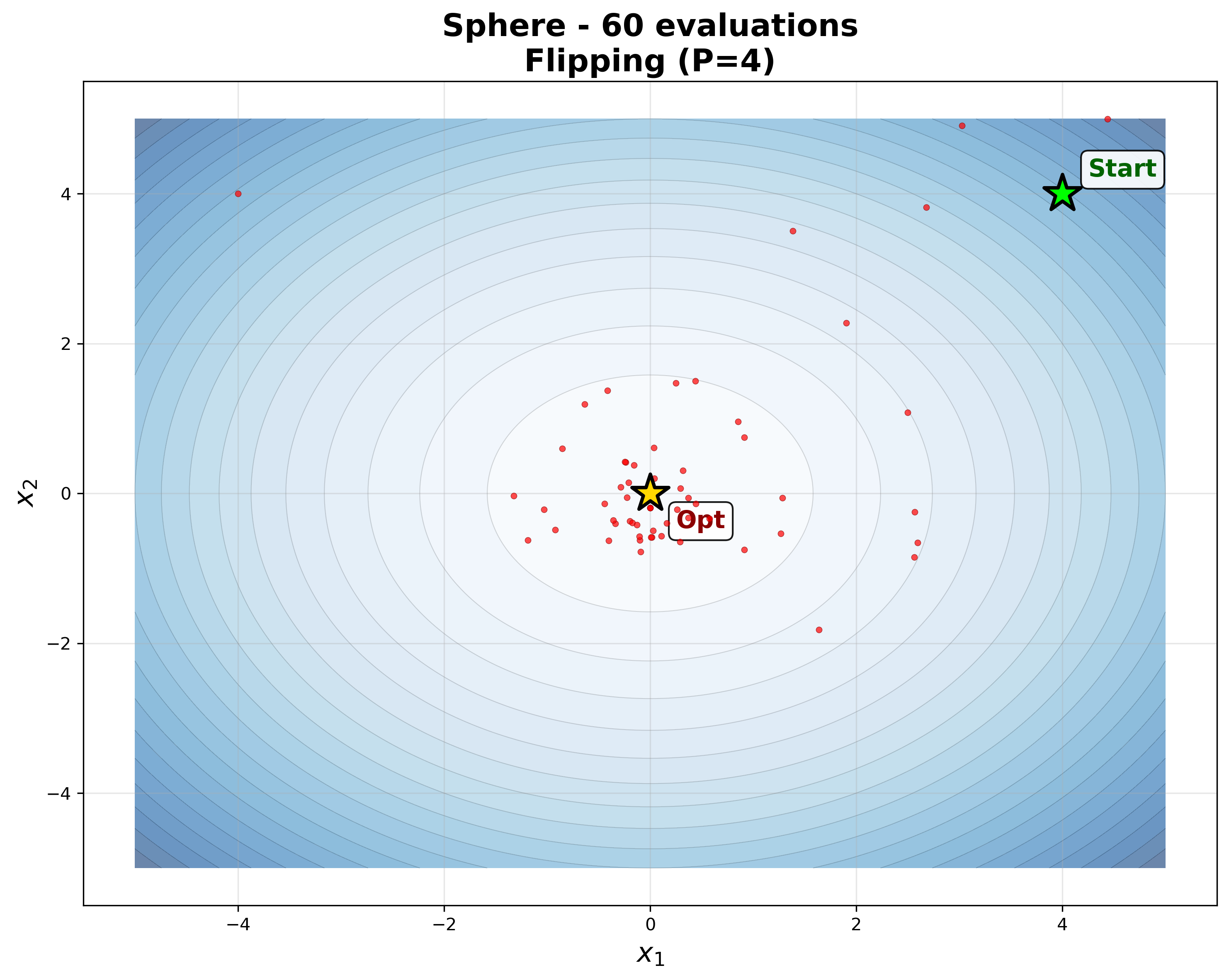}
  \end{minipage}
  \caption{Snapshots of the evaluated points (red dots) generated by the Flipping algorithm ($P=4$) on the 2D Sphere function at varying evaluation budgets ($N_f = 30, 40, 50, \text{and } 60$). The global optimum is marked by the yellow star at $(0,0)$.}
  \label{fig:scatter_evolution}
\end{figure}

\Cref{fig:scatter_evolution} captures the spatial distribution of all evaluated points (red dots) as the computational budget increases. In the initial phase (30 evaluations), the evaluated points are widely scattered across the contour space, reflecting the multi-directional exploration induced by the flipping mechanism. As iterations progress to 40 evaluations, newly generated points begin to cluster toward the central region, indicating that the updated quadratic models have captured the general descent direction. By 50 and 60 evaluations, the vast majority of points are densely concentrated around the global optimum $(0,0)$.

While basic, this visual test provides a straightforward confirmation of the algorithm's fundamental dynamics: it balances broad spatial exploration in the early stages with precise local exploitation as it converges to the minimum.

\begin{figure}[htbp]
  \centering
  \begin{minipage}{0.40\textwidth}
    \centering
    \includegraphics[width=\linewidth]{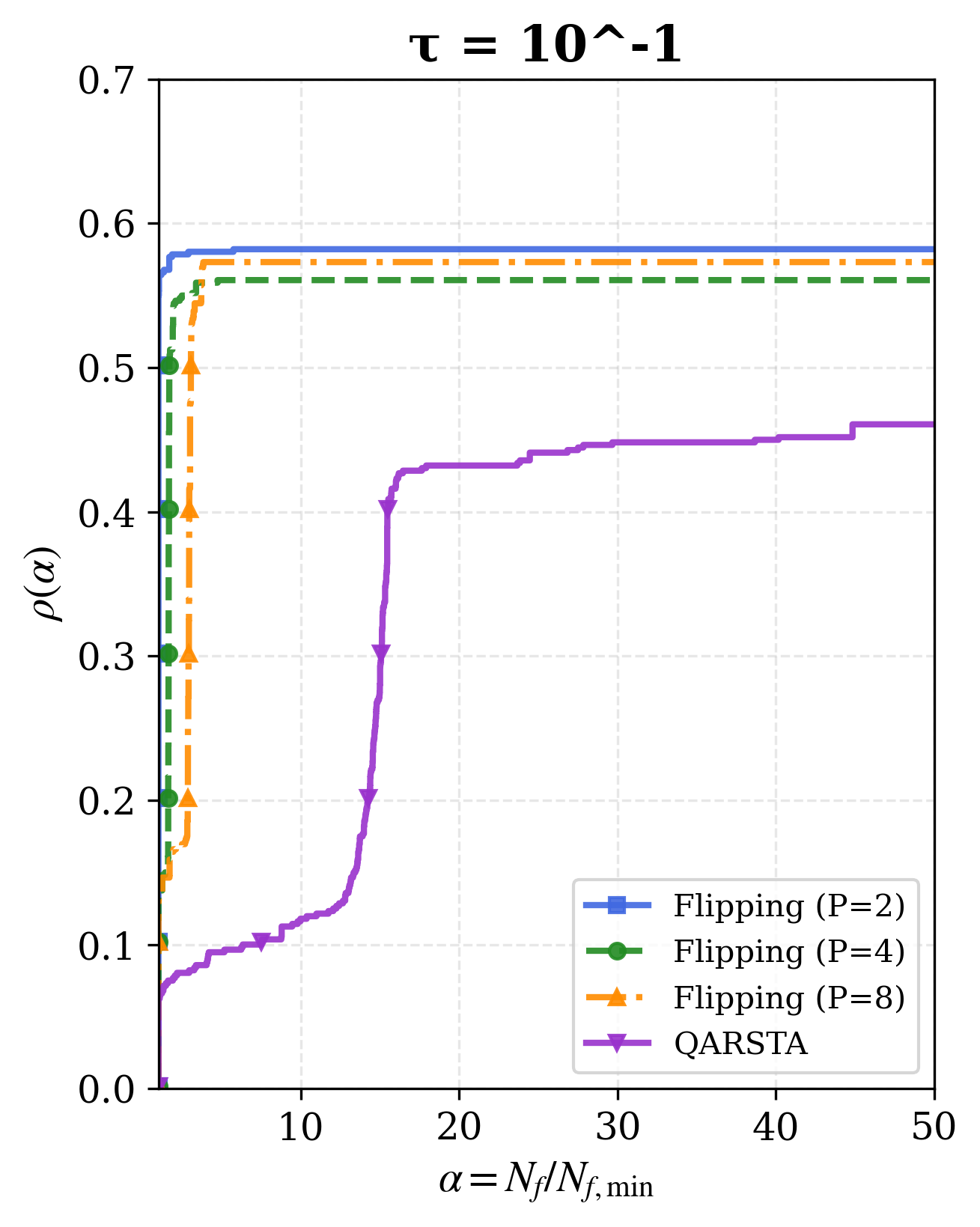}
  \end{minipage}\hfill
  \begin{minipage}{0.40\textwidth}
    \centering
    \includegraphics[width=\linewidth]{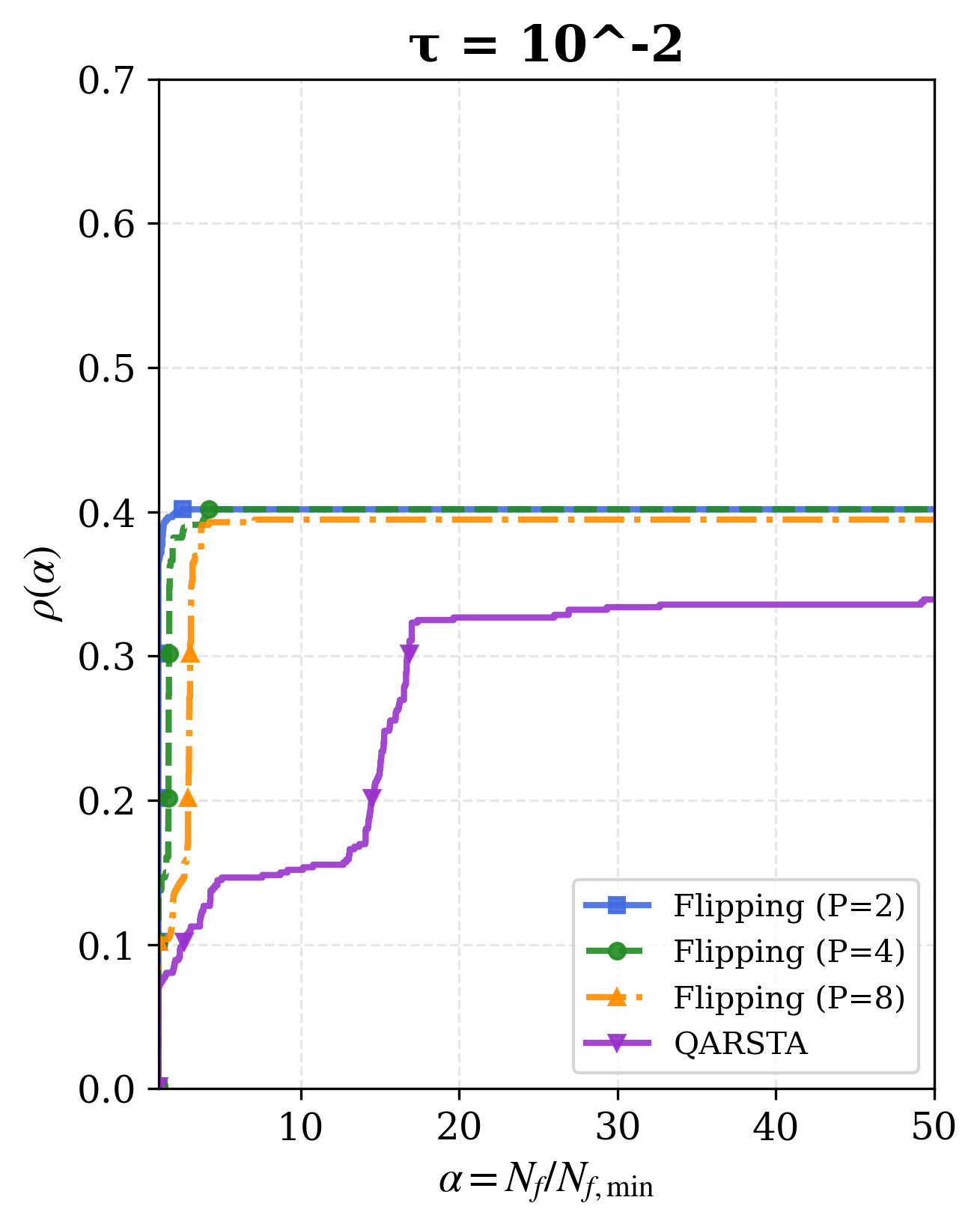}
  \end{minipage}
  \vspace{0.1cm}
  \begin{minipage}{0.40\textwidth}
    \centering
    \includegraphics[width=\linewidth]{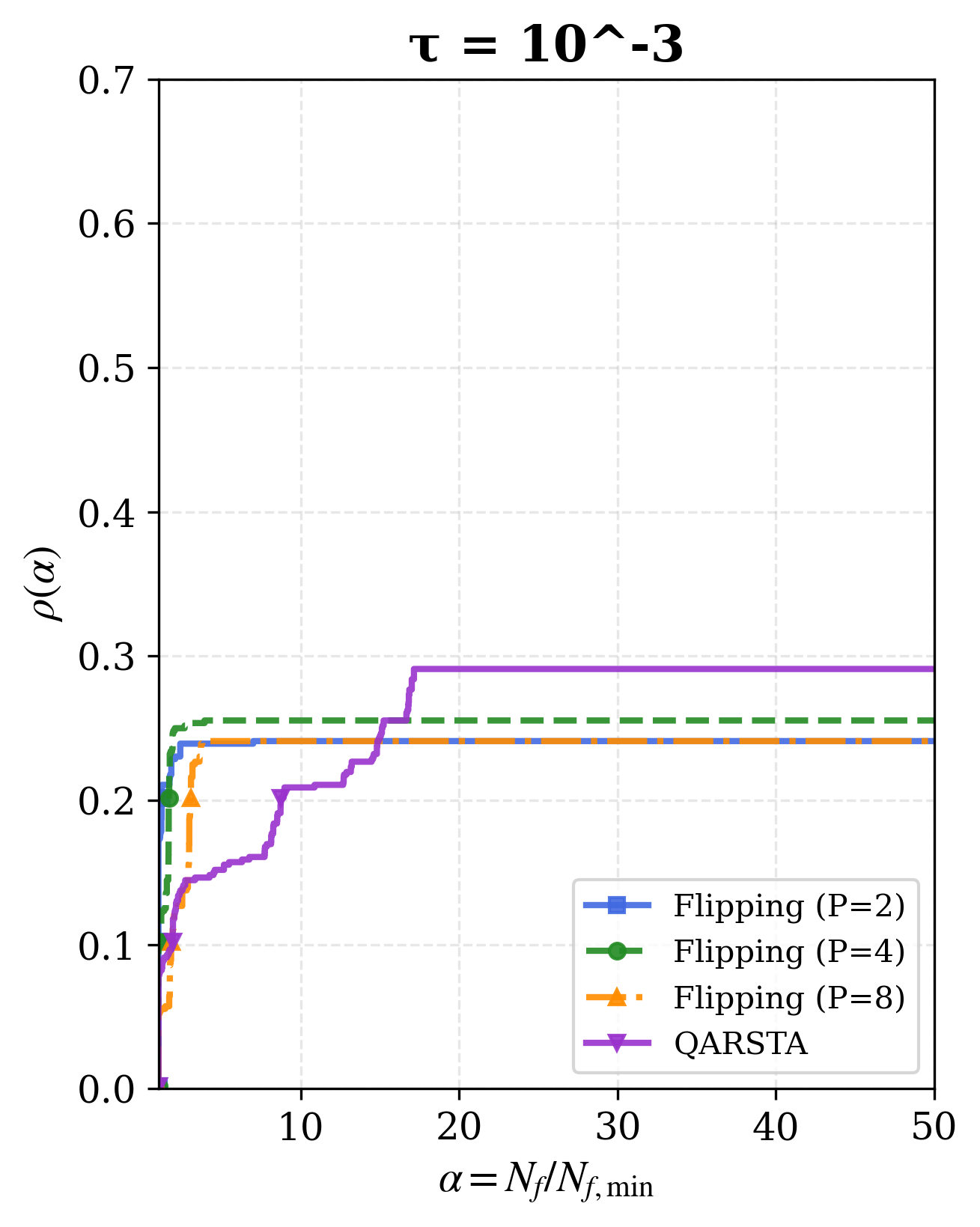}
  \end{minipage}\hfill
  \begin{minipage}{0.40\textwidth}
    \centering
    \includegraphics[width=\linewidth]{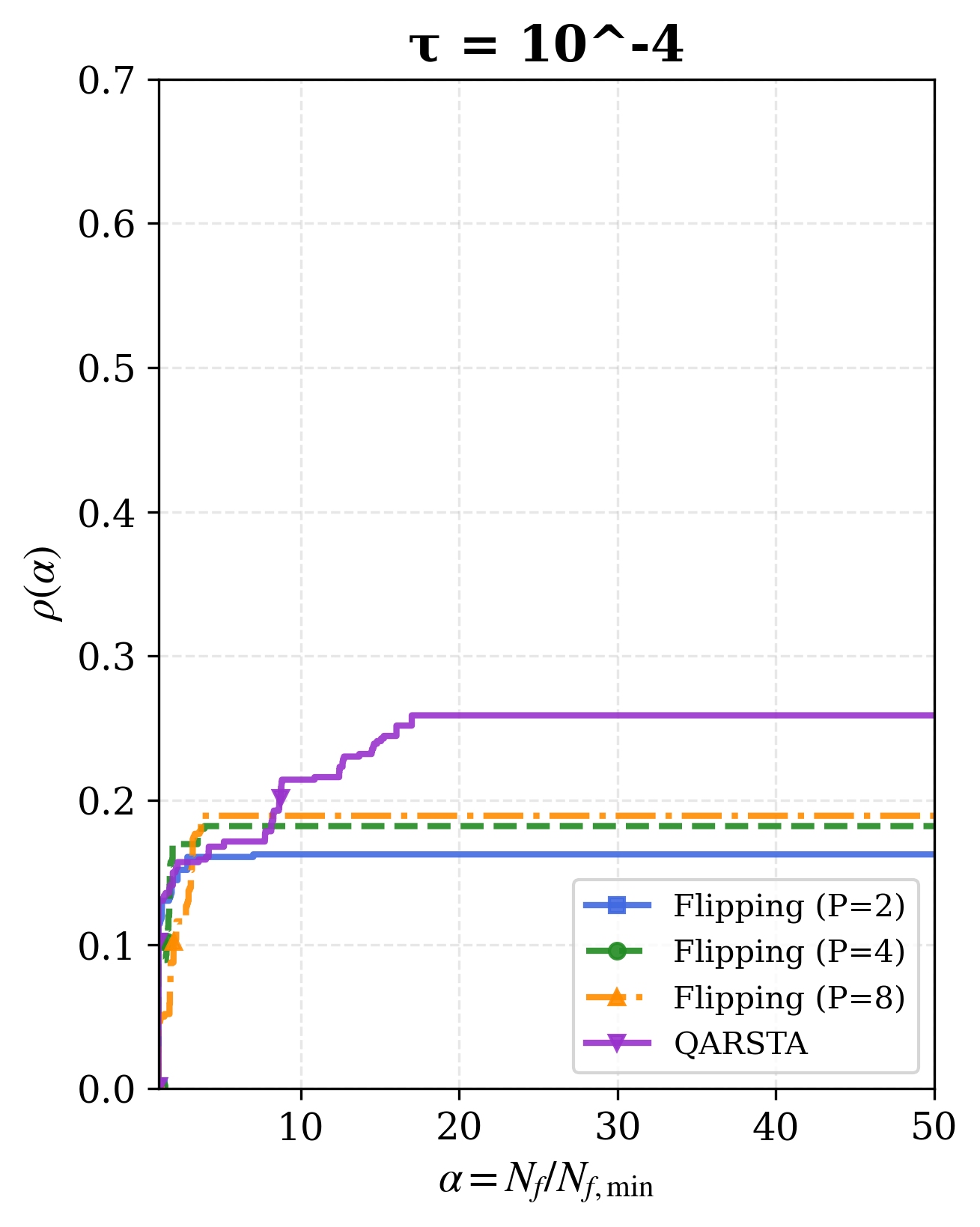}
  \end{minipage}
  \caption{Performance profiles $\rho(\alpha)$ for the 530 test problems under varying accuracy tolerance levels $\tau$. The x-axis indicates the relative performance ratio $\alpha = N_f/N_{f,\min}$.}
  \label{fig:perf_profiles}
\end{figure}

\Cref{fig:perf_profiles} presents the performance profiles for the tested algorithms. The results reveal a notable trade-off between the scale of parallel exploration ($P$) and the evaluation-based efficiency, which shifts depending on the strictness of the tolerance $\tau$.

For low to moderate precision requirements ($\tau = 10^{-1}$ and $10^{-2}$), the Flipping variant with $P=2$ achieves the highest y-intercept, indicating it has the greatest probability of being the most efficient solver. However, as the precision requirement increases ($\tau = 10^{-3}$), the $P=4$ configuration begins to outperform $P=2$. Under the most stringent tolerance ($\tau = 10^{-4}$), Flipping with $P=8$, despite requiring more evaluations initially, eventually surpasses both $P=2$ and $P=4$ in overall efficiency.

This phenomenon highlights the practical dynamics of distributed DFO. When locating a rough descent direction is sufficient (low precision), limiting the parallel breadth ($P=2$) prevents the algorithm from wasting function evaluations on secondary descent paths. Conversely, in the late stages of high-precision optimization, objective functions often exhibit flat regions or narrow valleys. Here, generating a wider variety of orthogonal flipped models ($P=8$) increases the likelihood of capturing subtle descent directions, compensating for the higher evaluation cost per iteration.

\begin{figure}[htbp]
  \centering
  \begin{minipage}{0.40\textwidth}
    \centering
    \includegraphics[width=\linewidth]{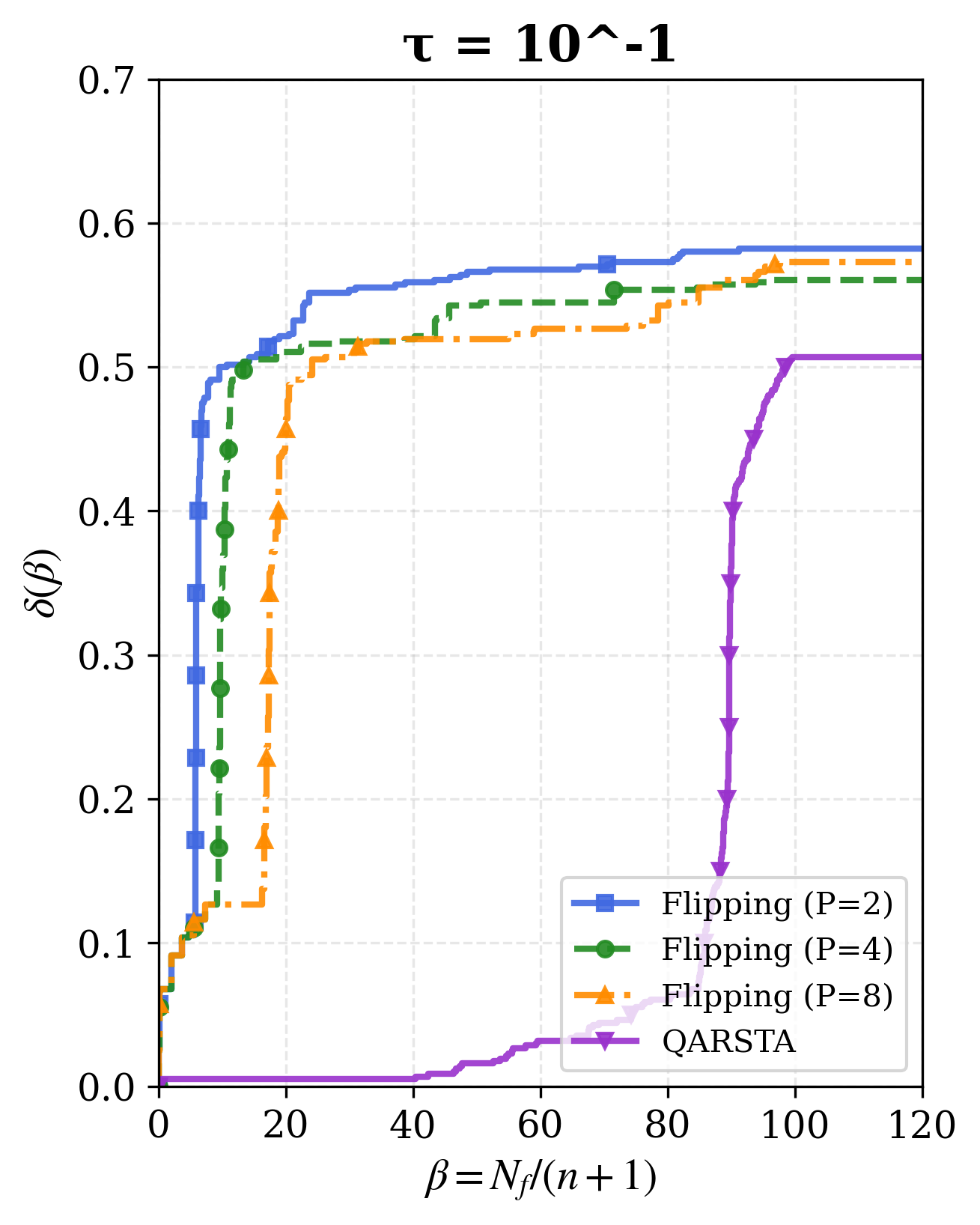}
  \end{minipage}\hfill
  \begin{minipage}{0.40\textwidth}
    \centering
    \includegraphics[width=\linewidth]{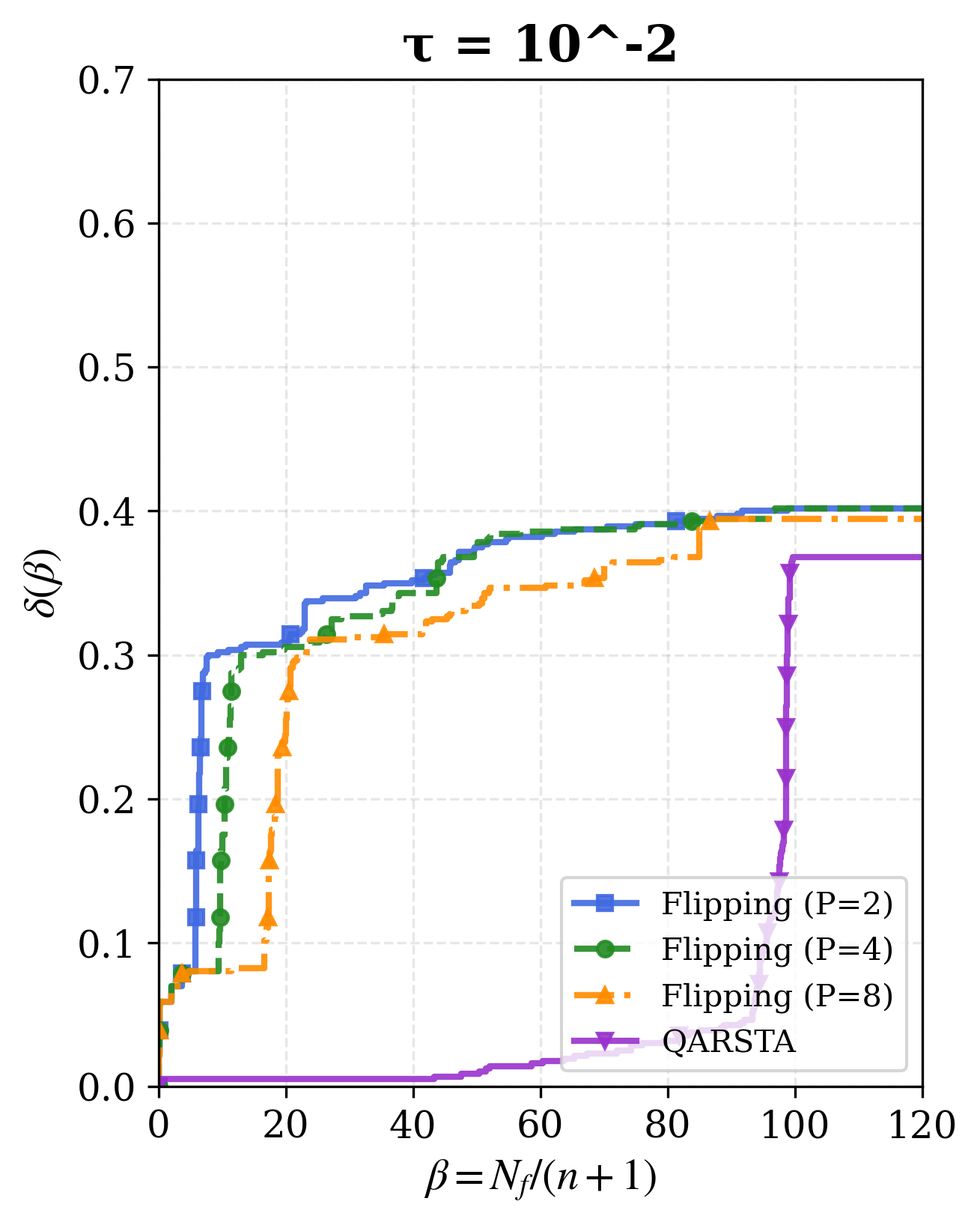}
  \end{minipage}
  \vspace{0.1cm}
  \begin{minipage}{0.40\textwidth}
    \centering
    \includegraphics[width=\linewidth]{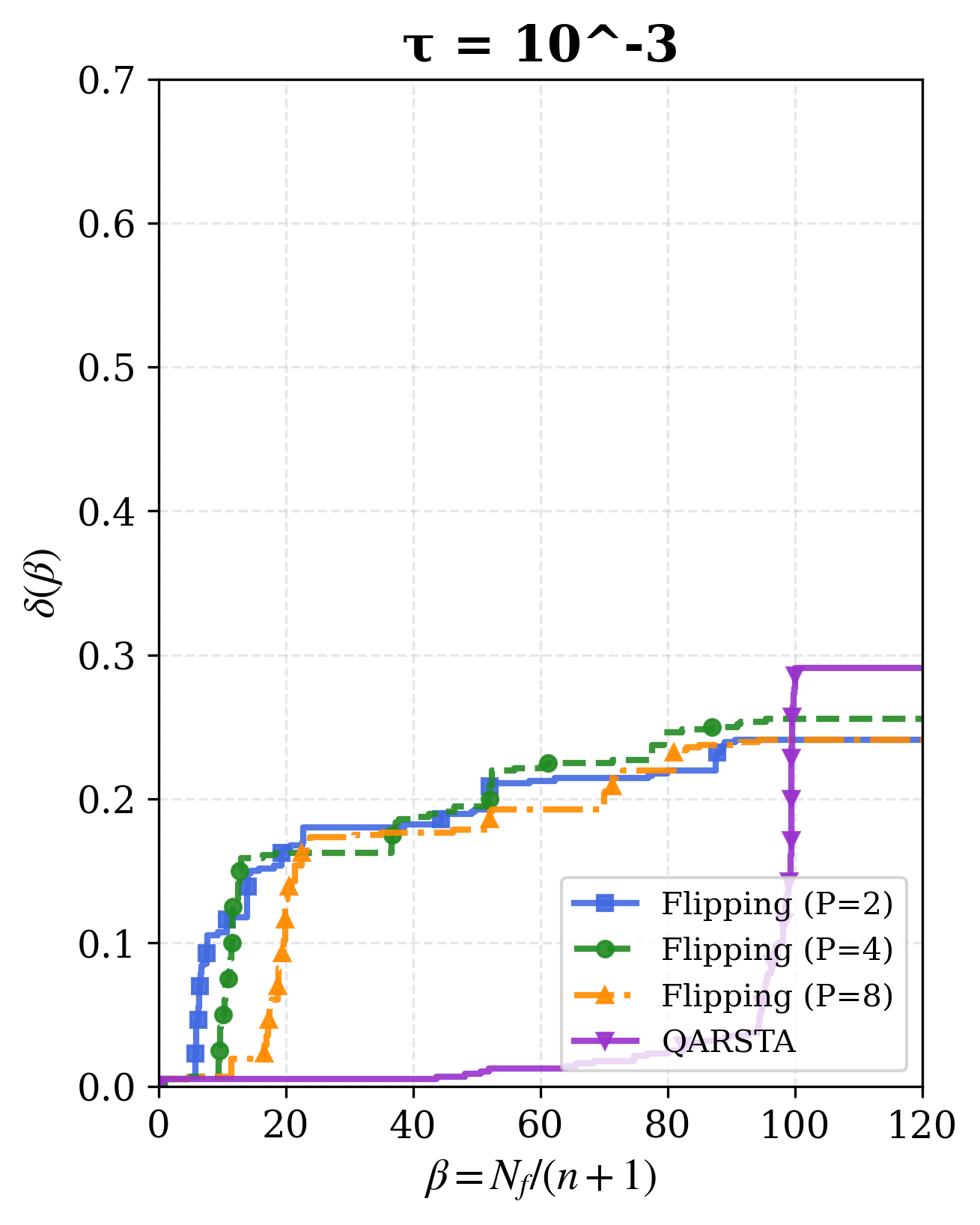}
  \end{minipage}\hfill
  \begin{minipage}{0.40\textwidth}
    \centering
    \includegraphics[width=\linewidth]{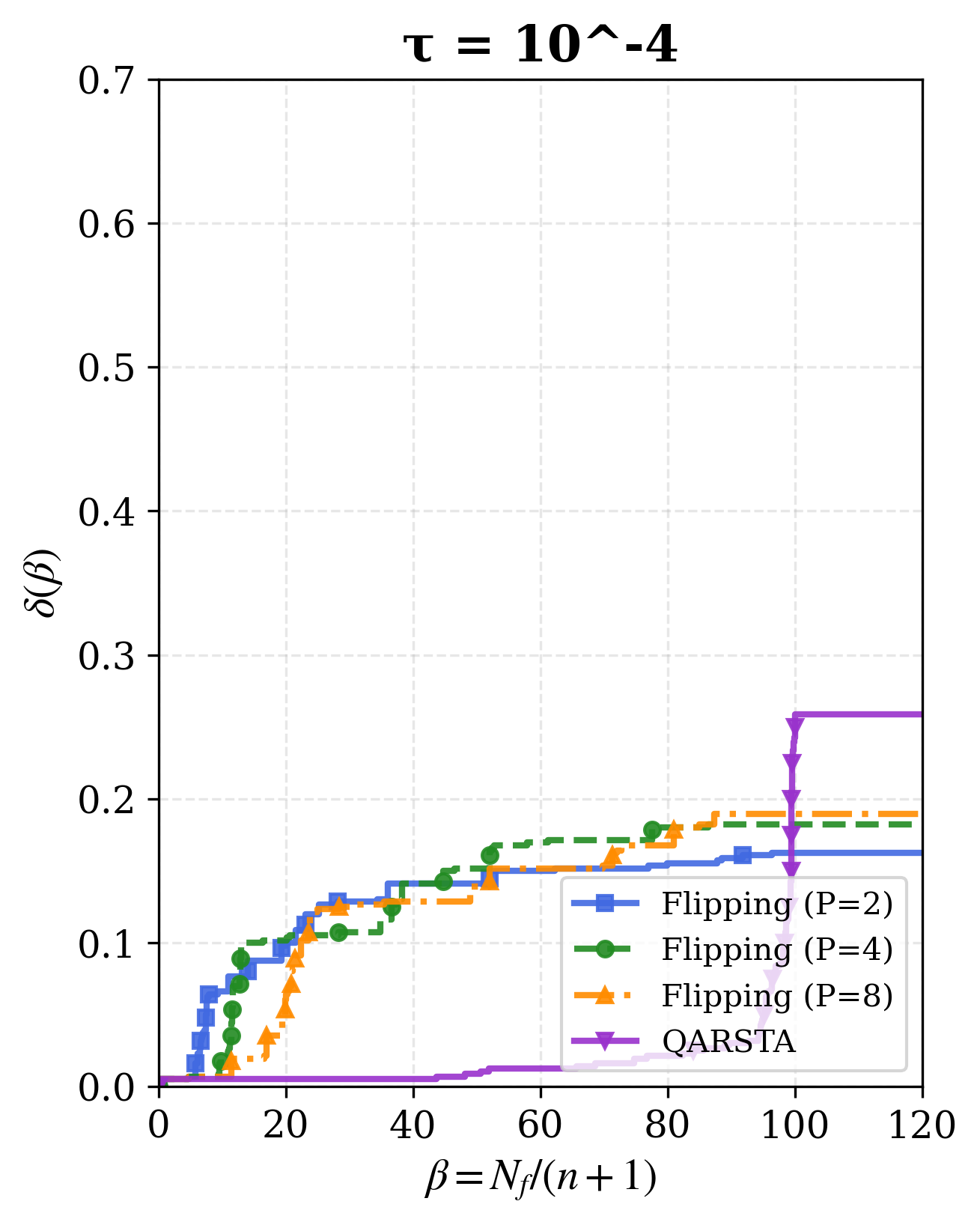}
  \end{minipage}
  \caption{Data profiles $\delta(\beta)$ for the 530 test problems under varying accuracy tolerance levels $\tau$. The x-axis indicates the normalized budget $\beta = N_f/(n+1)$.}
  \label{fig:data_profiles}
\end{figure}

The data profiles in \cref{fig:data_profiles} evaluate the robustness of the algorithms under limited computational budgets. A distinct characteristic of the proposed Flipping mechanisms is their steep initial ascent in the early-budget regime.

For highly restricted normalized budgets ($\beta \le 20$), the success rates of the Flipping algorithms climb rapidly. Even under the stringent condition of $\tau = 10^{-4}$, the Flipping variants quickly solve over $15\%$ of the instances while competing methods show minimal progress. This early-budget efficiency is consistent with the reduced cost of maintaining quadratic models, directing the TCG solver toward promising regions without requiring a prolonged initialization or stabilization phase.

\subsection{Performance Evaluation}
Comparing the proposed framework with baseline QARSTA reveals two distinct convergence behaviors.

QARSTA shows strong asymptotic properties. In high-precision scenarios ($\tau = 10^{-3}, 10^{-4}$), as the budget nears its limit ($\beta \to 100$), the success rate of QARSTA rises sharply, eventually surpassing the Flipping algorithms on this test set. This suggests QARSTA refines local models well when given an ample evaluation budget.

In real-world black-box optimization, however, each function evaluation can correspond to a costly physical simulation, making large budgets impractical. For most of the optimization process ($\beta \le 80$ or $\alpha \le 15$), QARSTA solves few problems in the early stages, while the Flipping algorithm maintains a substantial lead through the early-to-medium budget phases. By trading a modest degree of asymptotic convergence at extreme budgets, the Parallel Flipping mechanism provides a budget-robust solution suited for expensive large-scale optimization tasks.
\section{Conclusion}
\label{sec:discussion}

In this section, we discuss the mathematical advantages, practical performance boundaries, and future directions of the proposed framework.

The key advantage of the flipping mechanism is its algebraic-geometric duality. Traditional parallel derivative-free algorithms operate under a centralized paradigm: the master node decides the evaluation points, while worker nodes only evaluate the objective function in parallel. The master node still bears the sequential $\mathcal{O}(n^3)$ cost of reconstructing the surrogate models.

By contrast, the flipping mechanism supports \textit{decentralized model generation}. A coordinate reflection of the interpolation set preserves geometric poisedness in centered Euclidean trust regions (\cref{thm:poisedness}). This geometric invariance translates into a rank-at-most-two perturbation, letting each worker update the inverse KKT matrix in $\mathcal{O}(n^2)$ time when $m=\mathcal{O}(n)$. A valid surrogate model still requires objective values at the interpolation points used by that worker.

While the $\mathcal{O}(n^2)$ rank-2 update reduces the algebraic bottleneck, practical parallel efficiency on HPC clusters is governed by Amdahl's Law and the computation-to-communication ratio.

For low-dimensional problems ($n \le 100$), solving the KKT system already takes milliseconds. The overhead of MPI network communication (e.g., \texttt{MPI\_Allreduce} and broadcasting dense matrices) dominates the iteration time. In this regime, parallel speedup in wall-clock time is modest; the primary benefit of parallelization is a higher chance of escaping local minima.

In large-scale scenarios ($n \ge 1{,}000$), the centralized $\mathcal{O}(n^3)$ factorization can dominate computing time. The $\mathcal{O}(n^2)$ update may substantially reduce model-maintenance costs, though the net wall-clock benefit depends on objective evaluation cost, communication overhead, and the number of TCG iterations. Our framework includes an inner-loop strategy ($S=10$) to amortize the communication cost: workers perform multiple local trust-region steps and single-point low-rank updates before synchronizing.

The current approach has a structural limitation in its axis selection strategy. The flip axis $t_i$ is drawn uniformly from $\{1, \dots, n\}$. If the selected axis aligns with a flat direction of the objective function where the gradient component is near zero, flipping along this axis yields minimal functional variance and unproductive model exploration.

This limitation suggests several directions for future work. First, gradient-guided flipping: using the gradient $\nabla Q_k$ of the baseline model to bias the axis distribution toward coordinates with the largest absolute gradient components could focus parallel exploration on the steepest local descent directions. Second, integration with random subspaces: as the dimension $n$ scales to large values, flipping along standard coordinate axes may lose efficiency; recent work on expected decrease guarantees for random subspaces \cite{hare2024expected} and model-driven subspace reduction \cite{he2025modeldrivensubspaceslargescaleoptimization} suggests projecting the problem onto randomly generated orthogonal bases and applying the flipping mechanism within these lower-dimensional subspaces.

\label{sec:conclusions}

In this paper, we addressed the scalability bottleneck of model-based derivative-free optimization. Traditional solvers, limited by the $\mathcal{O}(m^3)$ linear algebra required to maintain underdetermined quadratic models, struggle to exploit modern parallel architectures. We introduced the Parallel Flipping Mechanism, which exploits the inner-product structure of the minimum Frobenius norm updating KKT system.

Our contributions span algebraic efficiency and geometric model management. We proved that reflecting interpolation geometry along a coordinate axis induces a rank-at-most-two perturbation on the KKT matrix. The Sherman-Morrison-Woodbury identity reduces the inverse KKT update to $\mathcal{O}(n^2)$ operations when $m=\mathcal{O}(n)$. We showed that this flipping operation is an isometry in centered Euclidean trust regions, preserving the poisedness constant. Together with standard fully linear model-management assumptions and valid objective values at the interpolation points, this supports first-order global convergence.

We integrated this mechanism with a Truncated Conjugate Gradient (TCG) subproblem solver in a decentralized Master-Worker topology. Numerical experiments on 530 benchmark problems show strong early-budget performance relative to QARSTA on the tested instances.

This work demonstrates how algebraic and geometric structure can be used to reduce sequential bottlenecks in model-based DFO. The proposed framework offers a theoretically grounded and scalable approach for large-scale, computationally expensive black-box optimization in distributed environments.

\bibliographystyle{plain}

\end{document}